\documentclass[reqno, A4paper]{amsart}
\usepackage{amsfonts, amsmath, amsthm, amssymb, latexsym, xfrac, mathrsfs, mathtools, float}

\usepackage{bbm}

\usepackage{tikz}
\usetikzlibrary{calc}

\usepackage[
  headheight=14pt,
  top=1.45in,bottom=1.35in,left=1.14in,right=1.58in,
  twoside,
  asymmetric,
]{geometry}
\usepackage{hyperref, xcolor}
\hypersetup{
    colorlinks = false,
    linkbordercolor = {white},
    pdfauthor=Lukas Langen
}

\usepackage{parskip}

\usepackage[font=scriptsize]{caption}
\usepackage{subcaption}

\newcommand\wt{\widetilde}
\newcommand\N{\ensuremath{\mathbb{N}}}
\newcommand\R{\ensuremath{\mathbb{R}}}

\renewcommand{\S}{\mathcal{S}}

\newcommand\C{\ensuremath{\mathbb{C}}}

\renewcommand\a{\mathfrak{a}}

\DeclareMathOperator{\Exp}{Exp}

\DeclareMathOperator{\End}{End}

\DeclareMathOperator{\diag}{diag}
\DeclareMathOperator{\conv}{conv}

\DeclareMathOperator{\Lie}{Lie}

\DeclareMathOperator{\g}{\mathfrak g}

\renewcommand{\k}{\mathfrak k}
\DeclareMathOperator{\p}{\mathfrak p}

\renewcommand{\Re}{\text{Re}}
\newcommand{\matr}[1]{\begin{pmatrix} #1 \end{pmatrix}}

\newcommand\opweylcham{\ensuremath{\a_{+}}}
\newcommand\weylcham{\ensuremath{\overline{\a_{+}}}}

\newcommand\areg{\ensuremath{\a_{\operatorname{reg}}}}
\newcommand\anreg{\ensuremath{\a^{I}_{\operatorname{reg}}}}

\DeclareMathOperator{\spann}{span}

\newcommand{\fundingThanks}{The author was supported by the German Research Foundation (DFG) via the grant SFB-TRR  358/1 2023-491392403.}

\makeatother

\usepackage[nameinlink,capitalize]{cleveref}

\makeatletter
\@namedef{subjclassname@2020}{\textup{2020} Mathematics Subject Classification}
\makeatother

\theoremstyle{plain}
\newtheorem{theorem}{Theorem}[section]
\newtheorem{corollary}[theorem]{Corollary}
\newtheorem{lemma}[theorem]{Lemma}
\newtheorem{sublemma}[theorem]{Sublemma}
\newtheorem{proposition}[theorem]{Proposition}
\theoremstyle{definition}

\theoremstyle{remark}
\newtheorem{remark}[theorem]{Remark}

\numberwithin{equation}{section}

\let\phi\varphi

\title{Uniform bounds on the Dunkl kernel}
\author{Lukas Langen}
\address{Institut f\"ur Mathematik, Universit\"at Paderborn, Warburger Str. 100, D-33098 Paderborn, Germany}
\email{llangen@math.upb.de}


\thanks{\fundingThanks}

\begin{document}

\date{\today}

\begin{abstract}
  For an arbitrary reduced root system, we give upper bounds for the Dunkl kernel with regular spectral parameter and its derivatives, which are uniform in the spatial variable. These estimates generalize well-known sharp upper bounds for classical one-variable Bessel functions and for spherical functions of Cartan motion groups.
  As a consequence, we prove that the representing measure of Dunkl's intertwining operator is absolutely continuous with respect to the Lebesgue measure for multiplicities $k> 1/2$ and generic spectral parameter. This settles a conjecture posed in \cite{roslerDeJeu} at least for $k>1/2$.
\end{abstract}

\maketitle

\section{Introduction}
Let $R$ be a reduced root system in a Euclidean space $(\a, \langle \cdot, \cdot\rangle)$ with finite reflection group $W$. Associated with the pair $(R, k)$, where $k\colon R\to\C$ denotes a $W$-invariant multiplicity function, Dunkl \cite{dunkl2} introduced a commuting family of differential-reflection operators.
Subsequently, these operators were studied (\cite{dunkl3, dunkl1991, opdamBessel}) and shown to possess joint eigenfunctions for prescribed spectral parameters $\lambda\in\a_{\C}$ and multiplicity $\Re\, k\geq 0$, given by the Dunkl kernel $\Exp_{k}(\lambda, \cdot)$.
To the Dunkl kernel, one can associate an integral transform - nowadays called the Dunkl transform - which intertwines the action of Dunkl operators with multiplication operators (\cite{dunkl5, dejeu}).
If $k=0$, the operators are just directional derivatives and the Dunkl kernel is the usual exponential function. On the other hand, for Weyl groups and certain half-integer multiplicities, the ($W$-invariant) Dunkl operators appear as radial parts of invariant differential operators on Cartan motion groups (\cite{heckman, deJeuPaleyWiener, dejeu}). In this case, the reflection group invariant analogue of the Dunkl kernel is a spherical function of this Cartan motion group (a Bessel function).
Consequently, the Dunkl transform shares similarities with both the classical Fourier transform on $\R^{n}$ and the spherical transform on Cartan motion groups (\cite{helgasonGroups}). 

In order to establish harmonic analysis for the Dunkl transform, such as an inversion and Plancherel theorem, estimates on the growth of the Dunkl kernel were of particular importance (\cite{dejeu}). While the rather rough estimates in \cite{dejeu} were sufficient for this task, sharp estimates and asymptotic expansions are still open. This is mainly due to irregular singularities of the Bessel system in contrast to the hypergeometric system, for which estimates and asymptotics of the solutions are known (\cite{hcSphericalI, hcSpherical, trombiVaradarajan, casselmanMilicic, gangolliVaradarajan, heckmanOpdamI, heckmanII, schapira,narayananPasqualePusti}). 
We postpone further comments on this issue to \cref{sec:remarks}.

For certain reflection groups, there are recent sharp asymptotic results for the Dunkl kernel, or at least for its symmetrized version (\cite{ankerTrojanDihedral, graczykSawyerA2, graczykSawyerAn}). However, these results rely heavily on the simple structure of the dihedral group $D_{2n}$, on explicit formulas for the Dunkl kernel associated with the symmetric group $\S_{3}$, or on recursive formulas for the (symmetrized) Dunkl kernel associated with $\S_{n}$, respectively.
For general root systems, the conjectured estimates are so far only established for sufficiently regular arguments, i.e. the estimates are uniform only in certain conical regions within an open Weyl chamber (\cite{roslerDeJeu, amriGasmi}).
These estimates were obtained by transforming the joint eigenvalue problem for the Dunkl kernel into a suitable first-order ODE system and a study of its asymptotic behaviour.
In the geometric cases of Cartan motion groups, sharp uniform upper bounds were obtained by Clerc (\cite{clercBessel}), even if the argument becomes singular. To this end, Clerc performed an inductive argument passing from regular arguments to progressively more singular ones.
The induction step is carried out by investigating the asymptotic properties of certain first-order systems satisfied by the spherical functions. 
There are also recent developments in the complex geometric cases, establishing uniform upper bounds in both the spectral and the spatial variables (\cite{etingofRains}). However, these bounds do not incorporate the structure of the root system and thus cannot be sharp.
Finally, we mention \cite{dziubanskiHejna}, where bounds for the Dunkl heat kernel for general root systems were obtained. They imply estimates for the Dunkl kernel, but these are not optimal.

In the present paper, we give, for fixed spectral parameter $\lambda\in\a+i\areg$ and multiplicity $\Re\, k\geq 0$, upper estimates for the Dunkl kernel $\Exp_{k}(\lambda,\cdot)$ which are uniform in the spatial argument, even if it becomes singular. We also obtain uniform upper bounds in the case $\lambda\in\areg$ if $k\geq 0$. Results in the aforementioned special cases suggest that these are optimal.
While the bounds in \cite{ankerTrojanDihedral, graczykSawyerA2, graczykSawyerAn} need multiplicity $k\geq 0$, our approach also works for multiplicities $\Re\, k \geq 0$. Furthermore, we even obtain uniform upper bounds on arbitrary derivatives $p(\partial)\Exp_{k}(\lambda, \cdot)$ of the Dunkl kernel, which, to our knowledge until now, were only established in the geometric cases (\cite{clercBessel}) and in the rank one case $R=A_{1}$. 

Our approach is as follows: Perform an induction to make regular arguments more singular, similar to \cite{clercBessel}. To this end, we investigate certain first-order systems associated with parabolic subgroups of $W$ coming from the Dunkl kernel eigenvalue property. For the trivial subgroup, the corresponding system recovers that of \cite{roslerDeJeu}. A decisive part of the proof, already in the situation of \cite{clercBessel}, is to further modify these systems so that they allow for a reasonable asymptotic study. However, the methods of \cite{clercBessel} extensively used the ambient geometric structure to do so and do not generalize to the Dunkl setting. We thus have to defer to completely different methods.
We postpone further discussion to \cref{sec:remarks}.

For $k\geq 0$ the Dunkl kernel admits an integral representation (\cite{roslerPositivity}) generalizing the Harish-Chandra integral formula for spherical functions (\cite{hcSphericalI, hcSpherical}), i.e. for $\lambda\in\a$ we have
\begin{align}\label{prop:repMeasure}
  \Exp_{k}(\lambda, z) = \int_{\a} e^{\langle z, \xi\rangle}\ d\mu^{k}_{\lambda}(\xi)\qquad(\forall z\in\a_{\C}) 
\end{align}
with a unique compactly supported probability measure $\mu^{k}_{\lambda}$. In the geometric cases, general results on pushforward measures immediately show that the $W$-invariant analogues of the measures $\mu^{k}_{\lambda}, \lambda\in\areg$ are absolutely continuous with respect to Lebesgue measure in $\a$, see \cite{helgasonGroups}. In fact, in the complex cases, corresponding to $k=1$, these measures are known as the so-called Duistermaat-Heckman measures, which are intensely studied in the context of compact Lie groups and more generally Hamiltonian actions (\cite{heckmanProjectionsOfOrbits, duistermaatHeckman1982variation, guilleminLermanSternberg}). In these cases, the density is piecewise polynomial.

In the general Dunkl setting, even the existence of the integral representation (\ref{prop:repMeasure}) was highly non-trivial as one lacks geometric arguments.
Still, it is conjectured (\cite{roslerDeJeu}) that for $k>0$ and $\lambda\in\areg$, the measure $\mu^{k}_{\lambda}$ is absolutely continuous with respect to Lebesgue measure in $\spann_{\R}R\subseteq\a$. This is known for root systems $R=A_{1}, A_{2}$ and for the $W$-invariant analogues in the case $R=A_{n}$ (\cite{roslerDeJeu, graczykSawyerA2,sawyer2017laplace}).
As an application of our estimates, we show that this conjecture is true for $k>1/2$ and an arbitrary reduced root system $R$. We furthermore obtain results on the differentiability of the density depending on $k$.

This paper is organized as follows. 
In \cref{sec:prelim}, we fix notations and state our main results. \cref{sec:ODE} concerns the construction of first-order systems associated with parabolic subgroups of $W$ from the Dunkl kernel eigenvalue property. \cref{sec:proofMainThm} consists of the proof of the main theorem. In \cref{sec:proofDeriv} and \cref{sec:proofCplx}, certain modifications in the proof are discussed in order to extend the result to derivatives of the Dunkl kernel and to not purely imaginary spectral parameters. \cref{sec:applications} is concerned with applications of the estimates, in particular with the absolute continuity of the measures $\mu^{k}_{\lambda}$. In \cref{sec:remarks}, we conclude with some comments on the geometric situation of Cartan motion groups and compare our situation with the case of Heckman-Opdam theory and symmetric spaces of non-compact type.

\section{Preliminaries and Results}\label{sec:prelim}

Let $(\a, \langle \cdot, \cdot\rangle)$ be a finite-dimensional Euclidean space with norm $|x|\coloneqq \sqrt{\langle x, x\rangle}$. 
We extend the inner product to a complex bilinear form on the complexification $\a_{\C}\coloneqq\C\otimes_{\R}\a$ of $\a$. 
Let $R\subseteq\a$ be a reduced root system spanning $\a$. We normalize the roots according to $\langle \alpha, \alpha\rangle =2$ for all $\alpha\in R$.
Denote by $W$ the finite reflection group generated by reflections $s_{\alpha}\colon x\mapsto x-\langle \alpha, x\rangle \alpha$ in the root hyperplanes $\alpha^{\perp}, \alpha\in R$. 
Fix a subset $R_{+}\subseteq R$ of positive roots with corresponding set of simple roots $\Delta_{+}\subseteq R_{+}$. Corresponding to these choices, we furthermore fix a positive Weyl chamber $\opweylcham = \{x\in\a : \langle \alpha, x\rangle  > 0\ (\forall \alpha\in R_{+})\}$. Finally, denote by $\areg\coloneqq\{x\in\a : \langle \alpha, x\rangle \neq 0\ (\forall \alpha\in R)\}$ the subset of regular elements in $\a$.

Fix a $W$-invariant function $k\colon R\to \C$, called the multiplicity function. Throughout this paper, we shall only be concerned with $\Re\, k\geq 0$, i.e. $\Re\, k_{\alpha}\geq 0\ (\forall\alpha\in R)$,  which we will assume from now on.

The Dunkl operator associated with $(R, k)$ into direction $\xi\in\a$ is given by
\begin{align*}
  T_{\xi}\coloneqq T_{\xi}(R, k) \coloneqq \partial_{\xi}+\sum_{\alpha\in R_{+}} k_{\alpha}\langle \alpha, \xi\rangle \frac{1-s_{\alpha}}{\langle \alpha, \cdot\rangle }. 
\end{align*}

The Dunkl operators define a commuting family of operators on $C^{2}(\a)$. Thus, we may define the homomorphism
\begin{align*}
  \C[\a]\to\End(C^{\infty}(\a)),\ p\mapsto p(T) = p(T(R, k))
\end{align*}
induced by $\a\ni \xi\mapsto T_{\xi}$. Here, $\C[\a]$ denotes the algebra of polynomial functions on $\a$, where we silently identify $\a\cong\a^{\ast}$ via $\langle \cdot, \cdot\rangle$. Similarly, we may identify $\C[\a]$ with constant coefficient differential operators by extending $\a\ni \xi\mapsto \partial_{\xi}$ to $\C[\a]\ni p\mapsto p(\partial)$.
The Dunkl operators are $W$-equivariant:
$$w\circ T_{\xi}\circ w^{-1}=T_{w\xi}.$$

Given $\lambda\in\a_{\C}$, consider the joint eigenvalue problem
\begin{align}\label{eqn:jointEVProb}
  \begin{cases}
    T_{\xi}f = \langle \xi, \lambda\rangle f\quad\text{for all}\ \xi\in \a,\\
    f(0)=1.
  \end{cases}
\end{align}
There is a unique holomorphic function $\Exp_{k}\colon\a_{\C}\times \a_{\C}\to \C$, called the Dunkl kernel, such that for all $\lambda\in\a_{\C}$, $\Exp_{k}(\lambda, \cdot)$ is the unique solution of (\ref{eqn:jointEVProb}), see \cite{dunkl1991, opdamBessel}. 

We recollect a few properties.
\begin{proposition}\label{prop:dunklKernel}
Let $w\in W$, $x,y \in\a_{\C}$ and $s\in \C$. Then
\begin{enumerate}
  \item $\Exp_{k}(sx, y)=\Exp_{k}(x, sy)$.
  \item  $\Exp_{k}(wx, wy)=\Exp_{k}(x, y)=\Exp_{k}(y, x)$.
  \item  $\overline{\Exp_{k}(x, y)}=\Exp_{\overline{k}}(\overline{x}, \overline{y})$.
\end{enumerate}
For $p\in\C[\a]$, there is a constant $C=C_{R,k, p}\geq 0$ such that
\begin{enumerate}
  \item[(4)] $|p(\partial)\Exp_{k}(x, y)|\leq C |x|^{\deg(p)}e^{\max_{w\in W} \Re \langle wx, y \rangle}$.
\end{enumerate}
Here, $p(\partial)$ is understood to act in the $y$-variable.
\end{proposition}
\begin{proof}
  Properties (1) and (3) are proven in \cite[Theorem 2.8]{dejeu}. Property (2) is due to the $W$-equivariance of the Dunkl operators and results of Dunkl (\cite{dunkl1991}); see also \cite[Theorem 2.8]{dejeu}. Estimate (4) is a consequence of Lemma 3.5 in \cite{dejeu}.
\end{proof}

The $W$-invariant analogue of the Dunkl kernel
\begin{align*}
  J_{k}(\lambda, \cdot) \coloneqq \frac{1}{|W|}\sum_{w\in W} \Exp_{k}(\lambda, w\cdot)
\end{align*}
is called the Bessel function. It can be characterized as the unique $W$-invariant holomorphic solution of the joint eigenvalue problem
\begin{align}\label{eqn:jointEVProbBessel}
  \begin{cases}
    p(T)f = p(\lambda) f\quad\text{for all}\ p\in \C[\a]^{W},\\
    f(0)=1.
  \end{cases}
\end{align}
Here, $\C[\a]^{W}$ denotes the algebra of $W$-invariant polynomials.

The goal of this paper is to improve the estimate in \cref{prop:dunklKernel}(4). While it is sharp for $k=0$, estimates in special cases \cite{ankerTrojanDihedral, graczykSawyerA2, graczykSawyerAn} and results for spherical functions of Cartan motion groups (\cite{clercBessel}; also see \cref{sec:remarks}) suggest better bounds. This was already speculated in \cite[4.3]{dunkl5} due to the asymptotic behaviour of classical Bessel functions, which appear in the rank one case (also see \cite[Section 2]{roslerDeJeu}). 

For multiplicities $\Re\, k\geq0$, which we assume, there is a rich theory based on an integral transform associated to $\Exp_{k}(\lambda, \cdot)$, called the Dunkl transform. In particular, there is a Plancherel theorem (\cite{dunkl5, dejeu}) establishing an isometric isomorphism $L^{2}(\a,\,|\omega_{k}| dx)\to L^{2}(\a ,\,|\omega_{k}| dx)$ with respect to the weight function
\begin{align*}
  \omega_{k}(x)\coloneqq \prod_{\alpha\in R_{+}}|\langle \alpha, x\rangle|^{2k_{\alpha}}.
\end{align*}
We introduce the modification
\begin{align*}
  \Omega_{k}(x)\coloneqq \prod_{\alpha\in R_{+}} (1+|\langle \alpha, x\rangle |)^{2k_{\alpha}}.
\end{align*}

Our first main result is

\begin{theorem}\label{thm:mainthm}
  Let $\lambda\in\areg$ and $\Re(k)\geq 0$. Then there exists a constant $C=C_{R,k,\lambda}\geq 0$ such that
  \begin{align}\label{ineq:mainthm}
    |\Exp_{k}(i\lambda, x)|\leq C\, |\Omega_{k}^{-1/2}(x)| 
  \end{align}
  for all $x\in\a$. 
\end{theorem}

In order to prove this theorem, we will start with sufficiently regular $x\in\a$ and inductively make them more singular. This will be done by transforming the joint eigenvalue problem of the Dunkl operators into a differential system of order one, where we will successively evolve into increasingly singular directions. Thus, we investigate the behaviour of $\Exp_{k}(\lambda, \cdot)$ along rays in singular directions starting from regular facets of certain sectors $S\subseteq\weylcham$.

Using the classical Phragmen-Lindelöf principle, we then conclude
\begin{corollary}\label{cor:realSpectralParam}
  Let $\lambda\in\opweylcham$ and $k\geq 0$. Then there exists a constant $C=C_{R, k,\lambda}\geq 0$ such that
  \begin{align}\label{ineq:mainCor}
    |\Exp_{k}(\lambda, x)|\leq C e^{\langle \lambda, x\rangle}\Omega_{k}^{-1/2}(x) 
  \end{align}
  for all $x\in\weylcham$. 
\end{corollary}

As it turns out, we can extend the approach of \cref{thm:mainthm} by an induction on the degree of $p\in\C[\a]$ and further obtain
\begin{theorem}\label{thm:mainThmDeriv}
  Let $\lambda\in\areg,\, \Re(k)\geq 0,$ and $p\in \C[\a]$. Then there exists a constant $C=C_{R,k,\lambda, p}\geq 0$ such that
  \begin{align*}
    |p(\partial)\Exp_{k}(i\lambda, x)|\leq C\, |\Omega_{k}^{-1/2}(x)|
  \end{align*}
  for all $x\in\a$. Here, $p(\partial)$ is understood to act in the $x$-variable.
\end{theorem}

With some modifications in the proof, we can further extend this result and obtain

\begin{theorem}\label{thm:mainThmComplex}
  Let $\lambda=\lambda_{1}+i\lambda_{2}\in \a+i\areg,\, \Re(k)\geq 0,$ and $p\in \C[\a]$. Then there exists a constant $C=C_{R,k,\lambda, p}\geq 0$ such that
  \begin{align*}
    |p(\partial)\Exp_{k}(\lambda, x)| \leq C\,  e^{\max_{w\in W} \Re \langle w\lambda, x\rangle } |\Omega_{k}^{-1/2}(x)|
  \end{align*}
  for all $x\in\a$. Here, $p(\partial)$ is understood to act in the $x$-variable.
\end{theorem}

\begin{remark}
  The assumption $\a=\spann_{\R}R$ poses no real restriction. Suppose $\a = \mathfrak{s}^{\perp}\oplus \mathfrak{s}$ with $\mathfrak s =\spann_{\R}R$.
  We can use the same arguments as in \cite[Remark 5]{roslerDeJeu} to see that the Dunkl kernel can be written as
  \begin{align*}
    \Exp^{\a}_{k}(\lambda_{0}+\lambda_{1}, x_{0}+x_{1})=e^{\langle \lambda_{0}, x_{0}\rangle }\Exp^{\mathfrak s}_{k}(\lambda_{1}, x_{1}),
  \end{align*}
  where $\Exp^{\mathfrak s}_{k}$ denotes the Dunkl kernel with respect to $(\mathfrak s, R, k)$ and $x_{0},\lambda_{0} \in \mathfrak{s}^{\perp}, x_{1},\lambda_{1}\in\mathfrak{s}$. Thus we obtain estimates for general reduced root systems by trivial modifications. Absolute continuity of the representing measures, however, can only be expected if the root system is spanning (see loc. cit.).
\end{remark}

\section{A system of order one}\label{sec:ODE}
We now transform the joint eigenvalue problem for the Dunkl kernel into a differential system of order one.

Throughout this section, let $H\in\weylcham$ be fixed. The stabilizer $W_{I}=\{w\in W : wH=H\}\subseteq W$ of $H$ in $W$ defines a standard parabolic subgroup, generated by simple reflections corresponding to a subset of simple roots $I\subseteq\Delta_{+}$. Accordingly, we define $R_{I}=\{\alpha\in R : \langle \alpha, H\rangle =0\}\subseteq R$. Note that $I\subseteq R_{I}$ is a simple positive system and $R_{+}\cap R_{I}$ is a corresponding set of positive roots in $R_{I}$.

We introduce the modified weight function
\begin{align*}
  \omega_{I}(x)\coloneqq \omega_{k, I}(x)\coloneqq\prod_{\alpha\in R_{+}\setminus R_{I}} |\langle \alpha, x\rangle|^{2k_{\alpha}}
\end{align*}
and furthermore define
\begin{align*}
  \Omega_{I}(x) \coloneqq \Omega_{k, I}(x)\coloneqq\prod_{\alpha\in R_{+}\cap R_{I}} (1+|\langle \alpha, x\rangle |)^{2k_{\alpha}}. 
\end{align*}
This $\Omega_{I}$ is complementary to $\omega_{I}$ in the sense that by multiplying them, we (almost) recover $\Omega_{k}$. It is worth noting that
\begin{align}\label{eqn:OmegaX0X}
  \Omega_{I}(x+tH) = \Omega_{I}(x) 
\end{align}
for all $t\in\R$. This will be relevant later.

Since $\langle \alpha, H\rangle =0$ for $\alpha\in R_{I}$, the Dunkl operator into direction $H$ degenerates:
\begin{align}\label{eqn:dunklOpParabolicDegen}
  T_{H} &= \partial_{H} + \sum_{\alpha\in R_{+}\setminus R_{I}} k_{\alpha}\langle \alpha, H\rangle \frac{1-s_{\alpha}}{\langle \alpha, \cdot\rangle }\nonumber\\
        &= \partial_{H} + \frac 1 2 \sum_{\alpha\in R\,\setminus\, R_{I}} k_{\alpha}\langle \alpha, H\rangle \frac{1-s_{\alpha}}{\langle \alpha, \cdot\rangle }.
\end{align}
In particular, this operator is $W_{I}$-invariant, as $W_{I}$ leaves the set $R\setminus R_{I}$ invariant.

We fix a spectral parameter $\lambda\in\a_{\C}$ and define
\begin{align}\label{eqn:defPhiFunction}
  \Phi(x) \coloneqq \Phi^{I,\lambda}(x)\coloneqq\big(\omega_{I}^{1/2}(x)e^{-\langle x, w\lambda\rangle}\Exp_{k}(w\lambda, x)\big)_{w\in W},
\end{align}
which is a differentiable function on
\begin{align*}
  \anreg \coloneqq \{x\in\a : \langle \alpha, x\rangle \neq 0\quad (\forall \alpha\in R\setminus R_{I})\}.
\end{align*}
We shall think of $\Phi$ as a $|W|\times 1$ column vector consisting of the entries 
$$F_{w}(x)\coloneqq F_{w}^{I,\lambda}(x)= \omega_{I}^{1/2}(x)e^{-\langle x, w\lambda\rangle}\Exp_{k}(w\lambda, x)$$
for $w\in W$.

Direct computation as in \cite[Lemma 1]{roslerDeJeu} shows
\begin{lemma}
  We have
  \begin{align*}
    \partial_{H} F_{w}(x) = \sum_{\alpha\in R_{+}\setminus R_{I}} k_{\alpha}\frac{\langle \alpha, H\rangle }{\langle \alpha, x\rangle}  e^{-\langle \alpha, x\rangle \langle \alpha, w\lambda\rangle }F_{s_{\alpha}w}(x) \qquad (x\in\anreg). 
  \end{align*}
\end{lemma}
We thus conclude

\begin{proposition}
  The function $\Phi$ solves the first-order system
  \begin{align}\label{eqn:matrixSystemPhi}
    \partial_{H} \Phi(x) = A(x)\Phi(x), 
  \end{align}
  on $\anreg$, where $A(x)=A^{H,\lambda}(x)=\big(A_{g, h}^{H,\lambda}(x)\big)_{g, h\in W}$ is the $\C^{|W|\times |W|}$-matrix valued function given by
  \begin{align*}
    A^{H,\lambda}_{g, h}(x) \coloneqq\begin{cases}
      k_{\alpha}\frac{\langle \alpha, H\rangle }{\langle \alpha, x\rangle }e^{-\langle \alpha, x\rangle \langle \alpha, g\lambda\rangle} &\text{if}\ h=s_{\alpha}g,\, \alpha\in R_{+}\setminus R_{I},\\
      0 &\text{otherwise.}
    \end{cases}
  \end{align*}
\end{proposition}

\begin{remark}
If $H\in\areg$ and $\lambda\in i\areg$, then the system (\ref{eqn:matrixSystemPhi}) coincides with the system of \cite[Corollary 3]{roslerDeJeu}. The authors there used the fact that this system is asymptotically zero in a suitable sense. Thus, by (a variant of) the Levinson theorem (\cite[Thm. 1.11.1]{eastham}), the non-zero solutions $\Phi(x)$ converge to a non-zero constant as $x\to\infty$ in sufficiently regular conical domains. The main conceptual difference in our paper is that, rather than evolving into regular directions, we investigate systems in increasingly singular directed rays. However, the poor integrability of $A$ will force us to further modify the system before we can use inductive arguments. 
\end{remark}

\section{Proof of \cref{thm:mainthm} and \cref{cor:realSpectralParam}}\label{sec:proofMainThm}

Throughout this section, we fix $\lambda\in \areg$ and $\Re(k)\geq 0$. Thus, all constants may a priori depend on $\lambda, k$ and the root system $R$, unless stated otherwise. 

We shall prove the following estimate:
\begin{align}\label{ineq:mainineq}
  \exists C\geq 0: |\Exp_{k}(i\lambda, x)| \leq C\,|\Omega_{k}^{-1/2}(x)|\qquad (\forall x\in\a).
\end{align}
Note that $\Omega_{k}$ is $W$-invariant, and it thus suffices to prove this on $\weylcham$.

Enumerate the simple roots $\Delta_{+}=\{\alpha_{1}, \ldots, \alpha_{n}\}$ and consider the sector
\begin{align*}
  S\coloneqq \{x\in \a : \langle \alpha_{1}, x\rangle \geq \ldots\geq \langle \alpha_{n}, x\rangle \geq 0\}\subseteq\weylcham. 
\end{align*}
For the rest of this section, the sector $S$ is fixed.
As there are only finitely many possible enumerations, and thus only finitely many such sectors $S\subseteq\weylcham$, it suffices to prove
\begin{align}\label{ineq:mainineqSector}
  \exists C\geq 0: |\Exp_{k}(i\lambda, x)| \leq C\,|\Omega_{k}^{-1/2}(x)|\qquad (\forall x\in S).
\end{align}

We consider the following subsets of $S$
\begin{align}\label{eqn:defai}
  \a^{0} \coloneqq \a^{0}_{S} &\coloneqq \{x\in S : \langle \alpha_{j}, x\rangle \leq 1\ (\forall 1\leq j\leq n)\}\\
  \a^{i}\coloneqq \a^{i}_{S} &\coloneqq \{x\in S : \langle \alpha_{1}, x\rangle =\ldots =\langle \alpha_{i}, x\rangle \geq 1\}
\end{align}
for $1\leq i \leq n$. Note that $S= \a^{0}\cup \a^{1}$, thus it suffices to establish (\ref{ineq:mainineqSector}) on $\a^{0}$ and $\a^{1}$ to conclude the main theorem. 
Since $\a^{0}$ is compact, we only need to prove an inequality of the form (\ref{ineq:mainineqSector}) on $\a^{1}$. This will be done via induction $i+1\to i$ from $i=n$ down to $i=1$. We thus define $\a^{n+1}\coloneqq\emptyset$ and consider the induction hypothesis:
\begin{align}\label{ineq:indHypothesis}
  \exists C\geq 0: |\Exp_{k}(i\lambda, x_{0})| \leq C\, |\Omega_{k}^{-1/2}(x_{0})|\quad (\forall x_{0}\in \a^{0}\cup \a^{i+1}).
\end{align}
By compactness of $\a^{0}$, the induction hypothesis is true for $i=n$.

We now perform the induction step, which will occupy the rest of this section. Thus, for the rest of this section, let $1\leq i \leq n$ be fixed.
Define $H_{i}\in S$ by
\begin{align*}
  \langle \alpha_{1}, H_{i}\rangle &=\ldots =\langle \alpha_{i}, H_{i}\rangle =1,\\
  \langle \alpha_{i+1}, H_{i}\rangle &= \ldots =\langle \alpha_{n}, H_{i}\rangle =0.
\end{align*}

Let $x\in \a^{i}$ and define 
\begin{align}\label{eqn:defx0}
   x_{0}\coloneqq x-t(x)H_{i}
\end{align}
with
\begin{align}\label{eqn:deftx}
  \begin{cases}
    t(x)\coloneqq\langle \alpha_{i}, x\rangle -\langle \alpha_{i+1}, x\rangle &\text{if}\ \langle \alpha_{i+1}, x\rangle \geq 1,\\
    t(x)\coloneqq \langle \alpha_{i}, x\rangle-1 &\text{if}\ \langle \alpha_{i+1}, x\rangle\leq 1.
  \end{cases}
\end{align}
In the case $i=n$, we choose the convention $\alpha_{n+1}=0$ and thus define $t(x)\coloneqq\langle \alpha_{n}, x\rangle-1$.
By definition $t(x)\geq 0$ and $x_{0}$ is a projection of $x\in \a^{i}$ into the set $\a^{0}\cup \a^{i+1}$. Indeed, $t(x)$ is chosen such that
\begin{align*}
  \begin{cases}
    \langle \alpha_{i+1}, x\rangle \geq 1 \implies x_{0} \in \a^{i+1},\\
    \langle \alpha_{i+1}, x\rangle \leq 1 \implies x_{0}\in \a^{0}.
  \end{cases}
\end{align*}
In any case, $t(x)$ is chosen in such a way that
\begin{align}\label{eqn:ax0}
  a \coloneqq a(x_{0})\coloneqq \langle \alpha_{1}, x_{0}\rangle  = \ldots = \langle \alpha_{i}, x_{0}\rangle \geq 1.
\end{align}
Thus if $x_{0}\in\a^{0}$, then actually $x_{0}\in \partial\a^{0}$. In particular, if $x\in\a^{n}$, then $x_{0}$ is the unique $x_{0}\in\a$ defined by $\langle \alpha_{1}, x_{0}\rangle =\ldots =\langle \alpha_{n}, x_{0}\rangle =1$. See \cref{fig:A2illustration} for an illustration in the setting $R=A_{2}\subseteq \a = \R^{3}_{0} = \{x\in\R^{3} : x_{1}+x_{2}+x_{3}=0\}$.

\begin{figure}[h]
\begin{tikzpicture}[scale=2.5, >=latex]
    \clip (-1.5,-1) rectangle (2.6,1.8);

    \coordinate (O) at (0,0);
    \coordinate (a1) at (-30:1);       
    \coordinate (a2) at (90:1);     
    \coordinate (a1a2) at (30:1);    
    
    \coordinate (w1) at (0:1/1.414);
    \coordinate (w2) at (60:1/1.414);

    \fill[red!30, opacity=0.3] (0,0) -- (0:4) arc (0:60:4) -- cycle;

    \draw[gray!30] (-1.5,0) -- (1.5,0);
    \draw[gray!30] (0,-1.5) -- (0,1.5);

    \draw[->, thick] (O) -- (a1) node[right] {$\alpha_1$};
    \draw[->, thick] (O) -- (a2) node[above left] {$\alpha_2$};
    \draw[->, thick] (O) -- (a1a2);
    
    \draw[->, thick] (O) -- (150:1);
    \draw[->, thick] (O) -- (270:1);
    \draw[->, thick] (O) -- (210:1);

    \draw[blue, dotted] (-150:1.5) node[below right] {\tiny $\alpha_1(x)=\alpha_2(x)$} -- (30:5) ;

    \draw[red, thick, opacity=0.4] (O) -- (0:4);
    \draw[red, thick, opacity=0.4] (O) -- (60:4);

    \node[red!60!black] at (45:1) {$\overline{\mathfrak a_{+}}$};

    \draw[darkgray, thick, opacity=0.4] (a1a2) -- ($(a1a2)+(-120:1/1.732)$);
    \fill[darkgray!30, opacity=0.3] (0,0) -- ($(a1a2)+(-120:1/1.732)$) -- (a1a2) -- cycle;
    \node[darkgray!60!black, align=center] at (16:0.5) {$\mathfrak a^{0}$};
    
    \fill[blue!20, opacity=0.65] ($(a1a2)+(-120:1/1.732)$)  -- (a1a2) -- ($10*(a1a2)$) -- (0:10) -- cycle;
    \node[blue!60!black, align=center] at (12.5:1.4) {$\mathfrak a^{1}$};

    \draw[green!80!black, thick] (a1a2) -- ($10*(a1a2)$);
    \node[green!65!black, align=center] at (28:2.6) {$\mathfrak a^{2}$};
    
    \coordinate (x) at (20:2.2);
    \node at (x) {\footnotesize$\bullet$};
    \node[above] at (x) {$x$};

    \coordinate (x0) at ($(x)-(0:0.75)$);
    \draw[<-, dotted] (x) -- (x0);
    \node at (x0) {\footnotesize $\bullet$};
    \node[above] at (x0) {$x_{0}$};


    \coordinate (xp) at (2.3,0.1);
    \node at (xp) {\footnotesize$\bullet$};
    \node[above] at (xp) {$x$};

    \coordinate (x0p) at ($(xp)-(0:1.67)$);
    \draw[<-, dotted] (xp) -- (x0p);
    \node at (x0p) {\footnotesize $\bullet$};
    \node[above right] at (x0p) {$x_{0}$};
\end{tikzpicture}
\caption{The sets $\a^{i}$ for $A_{2}\subseteq \R^{3}_{0}$}
\label{fig:A2illustration}

\end{figure}

We now construct the first-order system (\ref{eqn:matrixSystemPhi}) with respect to our fixed $H_{i}\in\weylcham$ and spectral parameter $i\lambda\in i\areg\subseteq\a_{\C}$.
By definition of $H_{i}$, we see that $I=\{\alpha_{i+1}, \ldots, \alpha_{n}\}\subseteq\Delta_{+}$ is the simple positive system generating the root system $R_{I}=\{\alpha\in R : \langle \alpha, H_{i}\rangle =0\}$ and thus the parabolic subgroup $W_{I} = \{w\in W : wH_{i}=H_{i}\}\subseteq W$.

We hence consider the system
\begin{align}\label{eqn:systemHi}
  \partial_{H_{i}} \Phi(x) = A(x)\Phi(x) \qquad (x\in \anreg),
\end{align}
where
\begin{align*}
  \Phi(x)=\Phi^{H_{i},i\lambda}(x) = (\omega_{I}^{1/2}(x)e^{-i\langle x, w\lambda\rangle} \Exp_{k}(iw\lambda, x))_{w\in W},
\end{align*}
and $A=A^{H_{i},i\lambda}$ is defined according to (\ref{eqn:matrixSystemPhi}).

\begin{lemma}\label{lem:estNonParabRoots}
  For each $\alpha\in R_{+}\setminus R_{I}$ there are constants $c_{\alpha}, C_{\alpha} > 0$ such that
  \begin{align*}
    c_{\alpha}(a+t) \leq \langle \alpha, x_{0}+tH_{i}\rangle \leq C_{\alpha}(a+t) \quad\text{for all}\quad t\geq 0.
  \end{align*}
  Here, $a\coloneqq a(x_{0})\coloneqq \langle \alpha_{1}, x_{0}\rangle =\ldots =\langle \alpha_{i}, x_{0}\rangle \geq 1$ (compare \cref{eqn:ax0}).
\end{lemma}
\begin{proof}
  Let $\alpha\in R_{+}\setminus R_{I}$.
  Since $\Delta_{+}\subseteq R_{+}$ is a simple subsystem, there are $c_{j}\geq 0$ such that $\alpha = \sum_{j=1}^{n}c_{j}\alpha_{j}$. But $R_{I}$ is a root system with simple system $I=\{\alpha_{i+1},\ldots, \alpha_{n}\}$, thus there exists an index $1\leq j \leq i$ such that $c_{j}>0$. Since $\langle \alpha_{j},x_{0}+tH_{i}\rangle=a+t$, this shows the first inequality. On the other hand, $x_{0}+tH_{i}\in S$ for all $t\geq 0$  and thus
  \begin{align*}
    \langle \alpha, x_{0}+tH_{i}\rangle \leq \Big(\sum_{j=1}^{n}c_{j}\Big) \langle \alpha_{1}, x_{0}+tH_{i}\rangle = \Big(\sum_{j=1}^{n}c_{j}\Big) (a+t). 
  \end{align*}
\end{proof}

As an immediate consequence, observe that $\langle \alpha, x_{0}+tH_{i}\rangle\geq c_{\alpha}>0$ and thus in particular $x_{0}+tH_{i}\in\anreg$ for all $t\geq 0$.

By abuse of notation, we now abbreviate 
\begin{align}\label{def:Phit}
   \Phi(t)\coloneqq \Phi^{H_{i},i\lambda}(x_{0}+tH_{i}) 
\end{align}
and similarly $A(t)=A^{H_{i},i\lambda}(x_{0}+tH_{i})$, so that
\begin{align}\label{eqn:system}
  \frac{d}{dt}\Phi(t)=A(t)\Phi(t)\quad (\forall t\geq 0). 
\end{align}

In analogy to \cite[Prop. 1]{roslerDeJeu}, we obtain
\begin{proposition}\label{prop:AQentries}\
  \begin{enumerate}
    \item  The matrix-valued improper Riemann integral $\int_{0}^{\infty}A(s)\, ds$\ converges, and\\ thus $Q(t)\coloneqq -\int_{t}^{\infty} A(s)\, ds$ is well-defined on $[0, \infty[$.
    \item The entries of $Q$ are linear combinations of terms of the following kind:
      \begin{align*}
        q_{\alpha, w}(t)\coloneqq k_{\alpha} (\pm i\operatorname{si}(\phi_{\alpha, w}(t))-\operatorname{Ci}(\phi_{\alpha, w}(t))),
      \end{align*}
      where
      \begin{align*}
        \phi_{\alpha, w}(t) = \langle \alpha, x_{0}+tH_{i}\rangle |\langle \alpha, w\lambda\rangle |
      \end{align*}
      and $\alpha\in R_{+}\setminus R_{I}$.
      
    \item  $AQ \in L^{1}([0, \infty[,\, \C^{|W|\times |W|})$. In fact, the entries of $AQ$ are linear combinations\\ of terms of the following kind:
      \begin{align*}
        I_{\alpha, \beta, w}(t)\coloneqq k_{\alpha}\frac{\langle \alpha, H_{i}\rangle}{\langle \alpha, x_{0}+tH_{i}\rangle} e^{-i\langle \alpha, x_{0}+tH_{i}\rangle\langle \alpha, w\lambda\rangle}(\pm i \operatorname{si}(\phi_{\beta, s_{\alpha}w}(t))-\operatorname{Ci}(\phi_{\beta, s_{\alpha}w}(t))), 
      \end{align*}
      with $\alpha, \beta\in R_{+}\setminus R_{I}$.
    \end{enumerate}
    Here, $\operatorname{si}(\tau)= -\int_{\tau}^{\infty}\frac{\sin u}{u}\, du$ and $\operatorname{Ci}(\tau)=-\int_{\tau}^{\infty}\frac{\cos u}{u}\, du$ denote trigonometric integrals.
\end{proposition}
\begin{proof}
  This can be proven analogously to \cite[Prop. 1]{roslerDeJeu}, as $a=a(x_{0})\geq 1$. Indeed, using the estimates $|\operatorname{si}(\tau)|\leq 2/\tau, |\operatorname{Ci}(\tau)|\leq 2/\tau$ (see \cite[(3.11)]{roslerDeJeu}) and the lower bound of \cref{lem:estNonParabRoots}, we see that the entries of $Q$ are bounded by $\frac{\text{const.}}{a+t}$ and those of $AQ$ are bounded by $\frac{\text{const.}}{(a+t)^{2}}$.
\end{proof}

We record the estimates from this proof for later use:
\begin{lemma}\label{lem:estNonParabRootsBetter}
  For all $w\in W$ and $\alpha, \beta\in R_{+}\setminus R_{I}$ 
  \begin{enumerate}
    \item    $\displaystyle |q_{\alpha, w}(t)| \leq \frac{\text{const.}}{a+t}.$
    \item  $\displaystyle |I_{\alpha, \beta, w}(t)| \leq \frac{\text{const.}}{(a+t)^{2}}.$
  \end{enumerate}
 \end{lemma}

 Now we argue as in the proof of a Levinson-type theorem, Thm. 1.11.1 in \cite{eastham}. This variant of the Levinson theorem immediately implies that $\Phi$ is asymptotically constant, i.e. $\Phi(t)$ converges as $t\to\infty$. However, in order to conclude the induction step, we need more precise knowledge of the limit in terms of an estimate involving $|\Omega^{-1/2}_{I}(x)|$. To this end, we shall reiterate parts of the proof of \cite[1.11.1]{eastham} and incorporate a bootstrap argument from \cite{clercBessel}. This bootstrap argument resembles the proof of the classical Gronwall lemma.

 Throughout the following, $\|\cdot\|$ denotes some fixed matrix norm on $\C^{|W|\times |W|}$.

\begin{lemma}\label{lem:QBound}
  There exists a constant $T\geq 0$ independent of $a=a(x_{0})$, such that $\|Q(t)\|\leq 1/2$ for all $t\geq T$. In particular, $\mathbbm{1}+Q(t)$ is invertible for all $t\geq T$ and $(\mathbbm{1}+Q(t))^{-1}$ is uniformly bounded for $t\geq T$. 
\end{lemma}
\begin{proof}
  Using \cref{lem:estNonParabRootsBetter}(1) and the fact that $a=a(x_{0})\geq 1$, we see that
  \begin{align*}
    \|Q(t)\| \leq \frac{\text{const.}}{1+t},
  \end{align*}
  with the constant depending only on $R, k$ and $\lambda$.
\end{proof}

Recall $\Phi$ from (\ref{def:Phit}) and note that
\begin{align}\label{ineq:OmegakOmegaI}
  |\Omega_{k}^{-1/2}(x_{0})| \leq |\omega_{I}^{-1/2}(x_{0})\Omega_{I}^{-1/2}(x_{0})|.
\end{align}

The induction hypothesis (\ref{ineq:indHypothesis}) thus entails the following:
\begin{lemma}[Good estimate for $\Phi(0)$]\label{lem:indHypothesisPhi}
  There exists a constant $C\geq 0$ such that
  \begin{align*}
    \|\Phi(0)\| \leq C\, |\Omega_{I}^{-1/2}(x_{0})|. 
  \end{align*}
\end{lemma}

Our next goal is to employ the operators $\mathbbm{1}+Q(t)$ to transform the system (\ref{eqn:system}) and improve the decay properties of the entries in the coefficient matrix, as seen in \cref{lem:estNonParabRootsBetter}(2) compared to \cref{lem:estNonParabRootsBetter}(1). However, since these operators become invertible only for $t\geq T$, we first have to extend \cref{lem:indHypothesisPhi} to obtain a bound on $\displaystyle\sup_{t\in [0, T]}\|\Phi(t)\|$.

\begin{lemma}\label{lem:OmegaNesting}
  There is a constant $c>0$ such that
  \begin{align*}
    c a^{-m_{I}} \leq |\Omega_{I}^{-1/2}(x_{0})| \leq 1 \leq a.
  \end{align*}
  Here, $a\coloneqq a(x_{0})\geq 1$ as before and\ \ $m_{I}\coloneqq \sum_{\alpha\in R_{+}\cap R_{I}}\Re(k_{\alpha})$.
\end{lemma}
\begin{proof}
First, note that
\begin{align*}
  |\Omega_{I}^{-1/2}(x_{0})| = \prod_{\alpha\in R_{+}\cap R_{I}}\frac{1}{(1+\langle \alpha, x_{0}\rangle)^{\Re(k_{\alpha})}}\leq 1. 
\end{align*}
On the other hand, we have $\langle \alpha, x_{0}\rangle \leq \text{const.}\cdot a$ for any $\alpha\in R_{+}$. Indeed, we can write $\alpha\in R_{+}$ as $\alpha=\sum_{j=1}^{n}c_{j}\alpha_{j}$ with $c_{j}\geq 0$. As $x_{0}\in S$, we have 
\begin{align*}
   \langle \alpha, x_{0}\rangle\ \leq\ \sum_{j=1}^{n}c_{j}\langle \alpha_{1}, x_{0}\rangle = C_{\alpha}a.
\end{align*}
Thus
\begin{align*}
  |\Omega_{I}(x_{0})| \leq \prod_{\alpha\in R_{+}\cap R_{I}} (1+C_{\alpha}a)^{2\Re(k_{\alpha})} \leq \text{const.}\cdot a^{2m_{I}}. 
\end{align*}
\end{proof}

In order to extend the bound of \cref{lem:indHypothesisPhi} to $\|\Phi(t)\|$ with $t\in [0, T]$, we use a bootstrap argument. To this end, we start with a very rough initial estimate on $\Phi$, which will be used later again.
\begin{lemma}[Initial estimate for $\Phi(t)$]\label{lem:initialEst}
  There is a constant $C\geq 0$ such that
  \begin{align*}
    \|\Phi(t)\|\leq C(a+t)^{\mu_{I}}\qquad (\forall t\geq 0).
  \end{align*}
  Here, $a=a(x_{0})\geq 1$ as before and\ \ $\mu_{I}\coloneqq \sum_{\alpha\in R_{+}\setminus R_{I}} \Re(k_{\alpha})$.
\end{lemma}
\begin{proof}
  This is an immediate consequence of \cref{prop:dunklKernel}(4) and the upper bound of \cref{lem:estNonParabRoots}. Indeed, note that
  \begin{align*}
    \|\Phi(t)\|=\|\Phi(x_{0}+tH_{i})\|\, \leq\, |\omega_{I}^{1/2}(x_{0}+tH_{i})|\ =\prod_{\alpha\in R_{+}\setminus R_{I}} |\langle \alpha, x_{0}+tH_{i}\rangle|^{\Re\, k_{\alpha}}. 
  \end{align*}
\end{proof}

This rough estimate will now be successively improved until we obtain:

\begin{lemma}[Good estimate for $\Phi(t)$ for $0\leq t\leq T$]\label{lem:indHypothesisPhiExtended}   
  There is a constant $C\geq 0$ such that
\begin{align*}
  \|\Phi(t)\| \leq C\, |\Omega_{I}^{-1/2}(x_{0})| 
\end{align*}
for all $t\in [0, T]$.
\end{lemma}
\begin{proof}
   As a consequence of the lower bound in \cref{lem:estNonParabRoots}, we see that the entries of $A(t)$ are bounded by $\frac{\text{const.}}{a}$. Now 
   \begin{align*}
      \Phi(t)=\Phi(0)+\int_{0}^{t}A(s)\Phi(s)\, ds,
   \end{align*}
   and using \cref{lem:indHypothesisPhi} and \cref{lem:initialEst}, we get for $t\in [0, T]$
\begin{align}\label{ineq:hdiIneq}
  \|\Phi(t)\| &\leq \|\Phi(0)\| +\int_{0}^{t} \|A(s)\|\|\Phi(s)\| \ ds\\
              &\leq\text{const.}\cdot|\Omega_{I}^{-1/2}(x_{0})| + \text{const.}\cdot\int_{0}^{t} \frac{(a+s)^{\mu_{I}}}{a} ds\nonumber\\
              &\leq \text{const.}\cdot \left(|\Omega_{I}^{-1/2}(x_{0})| + \frac{T(a+T)^{\mu_{I}}}{a}\right)\nonumber\\
              &\leq \text{const.}\cdot \left(|\Omega_{I}^{-1/2}(x_{0})| + a^{\mu_{I}-1}\right)\nonumber.
\end{align}
Recall that $T$ was independent of $a=a(x_{0})\geq 1$ and is thus treated as a constant. Using this new estimate in \cref{ineq:hdiIneq}, we further obtain
\begin{align*}
  \|\Phi(t)\| &\leq \text{const.}\cdot|\Omega_{I}^{-1/2}(x_{0})| + \text{const.}\cdot\int_{0}^{t} \frac{|\Omega_{I}^{-1/2}(x_{0})|+a^{\mu_{I}-1}}{a} ds\\
              &\leq \text{const.}\cdot (|\Omega_{I}^{-1/2}(x_{0})| + a^{\mu_{I}-2}).
\end{align*}
Iterating in this fashion, we eventually reach
\begin{align*}
  \|\Phi(t)\| \leq \text{const.}\cdot (|\Omega_{I}^{-1/2}(x_{0})| + a^{\nu}) 
\end{align*}
for $t\in [0, T]$, with some $\nu \leq -m_{I}$. Thus using the lower bound of \cref{lem:OmegaNesting} we obtain the lemma.
\end{proof}

To obtain bounds for $t\geq T$, we have to modify the system to guarantee better integrability.
Since $\mathbbm{1}+Q(t)$ is invertible for $t\geq T$, we may now consider the transformation 
\begin{align}
  Z(t)\coloneqq (\mathbbm{1}+Q(t))^{-1}\Phi(t) 
\end{align}
for $t\geq T$.

We perform a simple calculation, as in Eq. (1.11.6) in the proof of \cite[Thm. 1.11.1]{eastham} with $\Lambda = 0$ therein.
Recall that $Q'(t)=A(t)$. By differentiating the term $\Phi(t)=(\mathbbm{1}+Q(t))Z(t)$ and using the system (\ref{eqn:system}), we obtain
\begin{align}\label{eqn:derivationZtransform}
  A(t)Z(t)+(\mathbbm{1}+Q(t))Z'(t) = \Phi'(t) = A(t)(\mathbbm{1}+Q(t))Z(t). 
\end{align}
Rearranging terms, we thus conclude
\begin{align}\label{eqn:systemMod}
  \frac{d}{dt} Z(t) = (\mathbbm{1}+Q(t))^{-1} A(t)Q(t) Z(t)
\end{align}
for $t\geq T$. 

The major advantage of this system, in contrast to the system (\ref{eqn:system}), is that, by \cref{lem:estNonParabRootsBetter} and \cref{lem:QBound}, we obtain absolutely integrable entries in the coefficient matrix on the right side of (\ref{eqn:systemMod}). This allows for the following bootstrap argument.

\begin{lemma}[Initial estimate for $Z(t)$]\label{lem:initialEstMod}
  We keep the notation from \cref{lem:initialEst}. There is a constant $C\geq 0$ such that
  \begin{align*}
    \|Z(t)\|\leq C(a+t)^{\mu_{I}}\qquad (\forall t\geq T).
  \end{align*}
\end{lemma}
\begin{proof}
  This is an immediate consequence of \cref{lem:initialEst} and the uniform bound on $(\mathbbm{1}+Q(t))^{-1}$ obtained in \cref{lem:QBound}.
\end{proof}

\begin{lemma}[Good estimate for $Z(t)$ for $t\geq T$]\label{lem:bootstrapZ}
   There is a constant $C\geq 0$ such that
   \begin{align*}
     \|Z(t)\|\leq C\cdot |\Omega_{I}^{-1/2}(x_{0})|\qquad (\forall t\geq T). 
   \end{align*}
\end{lemma}
\begin{proof}
For $t\geq T$,
\begin{align}\label{eqn:intFormulaZ}
  Z(t) &= Z(T) + \int_{T}^{t} (\mathbbm{1}+Q(s))^{-1} A(s)Q(s)Z(s)\ ds.
\end{align}
Using the estimate of \cref{lem:estNonParabRootsBetter} for the coefficients of $AQ$ as well as our initial estimate for $Z$, we conclude that
\begin{align*}
  \|Z(t)\| &\leq \|Z(T)\| + \text{const.}\cdot \int_{T}^{t} \frac{(a+s)^{\mu_{I}}}{(a+s)^{2}}\ ds\\
           &\leq \|Z(T)\| + \text{const.}\cdot \int_{0}^{t} (a+s)^{\mu_{I}-2}\ ds.
\end{align*}
Note that
\begin{align*}
  \|Z(T)\|\leq \text{const.}\cdot |\Omega_{I}^{-1/2}(x_{0})| 
\end{align*}
as a consequence of \cref{lem:indHypothesisPhiExtended} and \cref{lem:QBound}.
Putting $\nu\coloneqq\mu_{I}$, we thus have
\begin{align}\label{ineq:Zintegral}
  \|Z(t)\|\leq \text{const.}\cdot \bigg(|\Omega_{I}^{-1/2}(x_{0})|+\int_{0}^{t}(a+s)^{\nu-2}\ ds\bigg). 
\end{align}

If $\nu> 1$, then 
\begin{align*}
  \|Z(t)\| &\leq \text{const.}\cdot \left(|\Omega_{I}^{-1/2}(x_{0})|+(a+t)^{\nu-1}\right).
\end{align*}
We can now use this new estimate in \cref{eqn:intFormulaZ} to obtain
\begin{align*}
  \|Z(t)\| &\leq \text{const.}\cdot \left(|\Omega_{I}^{-1/2}(x_{0})|+ \int_{0}^{t} \frac{|\Omega_{I}^{-1/2}(x_{0})| + (a+s)^{\nu-1}}{(a+s)^{2}}\ ds\right)\\
           &\leq \text{const.}\cdot \left(|\Omega_{I}^{-1/2}(x_{0})| + |\Omega_{I}^{-1/2}(x_{0})|\int_{0}^{\infty} (1+s)^{-2}\ ds + \int_{0}^{t} (a+s)^{\nu-3}\ ds \right)\\
           &\leq \text{const.}\cdot \left(|\Omega_{I}^{-1/2}(x_{0})| + (a+t)^{\nu-2}\right).
\end{align*}
We can progressively lower $\nu\to\nu-1$ in this fashion and reach
\begin{align}\label{ineq:ZiterationFinale}
  \|Z(t)\| &\leq \text{const.}\cdot \left(|\Omega_{I}^{-1/2}(x_{0})|+(a+t)^{\tau}\right)
\end{align}
with some exponent $0<\tau\leq 1$.
If $\nu = 1$ in \cref{ineq:Zintegral}, we calculate
\begin{align*}
  \int_{0}^{t}(a+s)^{\nu-2}\ ds = \log\left(\frac{a+t}{a}\right)\leq \text{const.}\cdot (a+t)^{1/2}.
\end{align*}
Thus, with this new estimate in \cref{eqn:intFormulaZ},
\begin{align}\label{ineq:Ziteration1Case}
  \|Z(t)\| &\leq \text{const.}\cdot \left(|\Omega_{I}^{-1/2}(x_{0})| + \int_{0}^{\infty}\frac{|\Omega_{I}^{-1/2}(x_{0})|}{(1+s)^{2}}\ ds + \int_{0}^{t} (a+s)^{-3/2}\ ds\right)\nonumber\\
           &\leq\text{const.}\cdot\left(|\Omega_{I}^{-1/2}(x_{0})| + a^{-1/2}\right).
\end{align}
After possibly one further iteration in \cref{eqn:intFormulaZ}, using \cref{ineq:ZiterationFinale} or \cref{ineq:Ziteration1Case}, we eventually reach
\begin{align*}
  \|Z(t)\|\leq \text{const.}\cdot \left(|\Omega_{I}^{-1/2}(x_{0})| + a^{\tau}\right)
\end{align*}
with some exponent $\tau\leq 0$. This finally leads to
\begin{align*}
\|Z(t)\|&\leq \text{const.}\left( |\Omega_{I}^{-1/2}(x_{0})| + \int_{0}^{\infty}\frac{|\Omega_{I}^{-1/2}(x_{0})|}{(a+s)^{2}}\ ds +\int_{0}^{\infty} a^{\tau}(a+s)^{-2}\ ds\right)\\
        &\leq \text{const.}\cdot \left(|\Omega_{I}^{-1/2}(x_{0})|+a^{\tau-1}\right).
\end{align*}
Recall the lower bound of \cref{lem:OmegaNesting}, which says that
\begin{align*}
  |\Omega_{I}^{-1/2}(x_{0})| \geq \text{const.}\cdot a^{-m_{I}}.
\end{align*}
Thus, if $\tau\leq -m_{I}$, we have
\begin{align*}
  \|Z(t)\|\leq\text{const.}\cdot |\Omega_{I}^{-1/2}(x_{0})| 
\end{align*}
and stop the process. Otherwise, we obtain the inequality
\begin{align*}
  \|Z(t)\|&\leq \text{const.}\left(|\Omega_{I}^{-1/2}(x_{0})| +\int_{0}^{\infty}\frac{|\Omega_{I}^{-1/2}(x_{0})|}{(a+s)^{2}}\ ds + \int_{0}^{\infty} a^{\tau-1}(a+s)^{-2}\ ds\right)\\
          &\leq \text{const.}\cdot \left(|\Omega_{I}^{-1/2}(x_{0})| + a^{\tau-2}\right).
\end{align*}
Successively lowering the exponent $\tau$ by $1$ until $\tau\leq -m_{I}$, we arrive at
\begin{align}\label{ineq:bootstrapZfinish}
  \|Z(t)\|&\leq\text{const.}\cdot (|\Omega_{I}^{-1/2}(x_{0})|+a^{\tau}) \leq \text{const.}\cdot |\Omega_{I}^{-1/2}(x_{0})|. 
\end{align}
This finishes the proof of \cref{lem:bootstrapZ}.
\end{proof}

We thus conclude
\begin{proposition}[Good estimate for $\Phi(t)$ for $t\geq 0$]\label{lem:bootstrapPhiFinish}
  There is a constant $C\geq 0$ such that
  \begin{align*}
    \|\Phi(t)\| \leq C\, |\Omega_{I}^{-1/2}(x_{0})|\qquad (\forall t\geq 0).
  \end{align*}
\end{proposition}
\begin{proof}
  This is immediate for all $t\geq T$ by combining \cref{lem:QBound} and \cref{lem:bootstrapZ}. For all $0 \leq t\leq T$ this was the content of \cref{lem:indHypothesisPhiExtended}. 
\end{proof}

Recall the definition of $t(x)$ from \cref{eqn:deftx}, which allowed us to write $x=x_{0}+t(x)H_{i}$ with $t(x)\geq 0$. Furthermore, recall \cref{eqn:OmegaX0X}, which implies that $\Omega_{I}(x_{0})=\Omega_{I}(x)$.
Then, putting $t=t(x)$ and considering the first entry of $\Phi$, we see that
\begin{align*}
  |\omega_{I}^{1/2}(x)\Exp_{k}(i\lambda, x)|&\leq \text{const.}\cdot |\Omega_{I}^{-1/2}(x)|.
\end{align*}
Hence
\begin{align*}
  |\Exp_{k}(i\lambda, x)| \leq \text{const.}\cdot |\omega_{I}^{-1/2}(x)\Omega_{I}^{-1/2}(x)| \leq C_{R, k,\lambda}\, |\Omega_{k}^{-1/2}(x)|.
\end{align*}
This concludes the inductive step $i+1\to i$, and thus finishes the proof of \cref{thm:mainthm}. 
\qed

\subsection{Proof of \cref{cor:realSpectralParam}}
We use a Phragmen-Lindelöf argument.
Suppose that $k\geq 0$. Then $\Theta_{k}(x)\coloneqq \prod_{\alpha\in R_{+}}(1+\langle \alpha, x\rangle)^{2k_{\alpha}}$ is holomorphic in an open neighbourhood of $\{x\in\a_{\C} : \Re\, x \in \weylcham\}$ and $|\Theta_{k}(x)|\leq \Omega_{k}(x)$. For fixed $\lambda\in\opweylcham$ and $x\in\weylcham$, consider the function
\begin{align*}
  G(z)\coloneqq \Theta_{k}^{1/2}(zx)e^{-\langle \lambda, zx\rangle}\Exp_{k}(\lambda, zx), 
\end{align*}
which is holomorphic in an open neighbourhood of $H=\{z\in\C : \Re\, z\geq 0\}$. Note that $\Re\langle w\lambda, zx\rangle \leq \langle \lambda, (\Re\, z)x\rangle$ for all $w\in W$ and $z\in H$, as $\lambda$ and $x$ are contained in the same positive Weyl chamber.
Using \cref{prop:dunklKernel}(4), we obtain
\begin{align*}
  |G(z)| \leq C\, \Omega_{k}^{1/2}(zx) 
\end{align*}
with a constant independent of $z\in H$. In particular, $G$ is of subexponential growth. Furthermore, by \cref{thm:mainthm}, $G$ is bounded on the imaginary axis. By the Phragmen-Lindelöf theorem (\cite[Thm. 5.61]{titchmarsh}), we therefore conclude that $G$ is bounded in $H$.
 Thus
\begin{align*}
  |\Exp_{k}(\lambda, x)| \leq C\, e^{\langle \lambda, x\rangle}\Omega_{k}^{-1/2}(x) \qquad (\forall x\in \weylcham)
\end{align*}
with some constant $C=C_{R, k,\lambda}\geq 0$.\qed

\section{Proof of \cref{thm:mainThmDeriv}}\label{sec:proofDeriv}
We prove this by induction on the degree $\deg(p)$. 
Throughout this section, let $\lambda\in\areg$ be fixed.
We mention that Clerc (\cite{clercBessel}) also performs an induction on the degree to establish estimates for derivatives in the geometric cases. However, this is easier in those cases, and as we have to deal with some reflection terms in the general case, we separated the proof of \cref{thm:mainThmDeriv} from that of \cref{thm:mainthm}. We postpone further comments on this issue to \cref{sec:remarks}.

  Note that \cref{thm:mainthm} establishes the claim for $\deg(p)=0$. 
  To deal with the case $\deg(p)=1$, consider $\partial_{\xi}$ with $\xi\in\a$. We repeat the arguments from the last section in a slightly modified form. That is, we establish bounds for $\partial_{\xi}\Exp_{k}(i\lambda, \cdot)$ on a fixed sector $S$ by induction from subsets $\a^{0}\cup\a^{i+1}$ to subsets $\a^{i}$. Here, the subsets $\a^{i}$ are those defined in \cref{eqn:defai}.

  Thus, let $x\in\a^{i}$ and again define $t(x)\geq 0$ and $x_{0}\in\a^{0}\cup\a^{i+1}$ as in \cref{eqn:defx0,eqn:deftx}, so that $x=x_{0}+t(x)H_{i}$ and $a\coloneqq a(x_{0})\coloneqq \langle \alpha_{1}, x_{0}\rangle =\ldots =\langle \alpha_{i}, x_{0}\rangle \geq 1$. We again consider the system (\ref{eqn:systemHi})
  \begin{align*}
    \partial_{H_{i}}\Phi^{H_{i},i\lambda}(x) = A^{H_{i},i\lambda}(x)\Phi^{H_{i},i\lambda}(x). 
  \end{align*}
  Now let $\wt\Phi(x) \coloneqq (e^{-i\langle x, w\lambda\rangle }\Exp_{k}(iw\lambda, x))_{w\in W}$, so that $\Phi(x)\coloneqq \Phi^{H_{i},i\lambda}(x)=\omega_{I}^{1/2}(x)\wt\Phi(x)$. Further, fix some $\xi\in\a$. We then calculate (abbreviating $H\coloneqq H_{i}$)
  \begin{align*}
    \partial_{\xi}\partial_{H}(\omega_{I}^{1/2}\wt\Phi) = \partial_{H}(\omega_{I}^{1/2}\partial_{\xi}\wt\Phi) + \partial_{H}((\partial_{\xi}\omega_{I}^{1/2}) \wt\Phi)
  \end{align*}
  and on the other hand
  \begin{align*}
    \partial_{\xi}(A\omega_{I}^{1/2}\wt\Phi) = (\partial_{\xi}A)\omega_{I}^{1/2}\wt\Phi + A(\omega_{I}^{1/2}\partial_{\xi} \wt\Phi)+ A(\partial_{\xi}\omega^{1/2}_{I})\wt\Phi. 
  \end{align*}
  We thus conclude that 
  \begin{align}\label{eqn:defPsi}
  \Psi(x) \coloneqq \omega_{I}^{1/2}(x)\partial_{\xi}\wt\Phi(x) = \Big(\omega_{I}^{1/2}(x)\partial_{\xi}\big(e^{-i\langle \cdot, w\lambda\rangle}\Exp_{k}(iw\lambda, \cdot)\big)(x)\Big)_{w\in W} 
  \end{align}
  satisfies the system
  \begin{align}\label{eqn:systemPsi}
    \partial_{H}\Psi(x) &= \partial_{\xi}\partial_{H}(\omega_{I}^{1/2}\wt\Phi) -\partial_{H}((\partial_{\xi}\omega_{I}^{1/2})\wt\Phi)\nonumber\\
                        &= \partial_{\xi}(A\omega_{I}^{1/2}\wt\Phi)-\partial_{H}((\partial_{\xi}\omega_{I}^{1/2})\wt\Phi)\nonumber\\
    &= A(x)\Psi(x) + (\partial_{\xi}A(x))\omega_{I}^{1/2}(x)\wt\Phi(x)+A(x)(\partial_{\xi}\omega_{I}^{1/2}(x))\wt\Phi(x)\nonumber\\
                        &\qquad\qquad-(\partial_{H}\partial_{\xi}\omega_{I}^{1/2}(x))\wt\Phi(x)-(\partial_{\xi}\omega_{I}^{1/2}(x))(\partial_{H}\wt\Phi(x))\nonumber\\
                        &=A(x)\Psi(x) + B_{1}(x)+ B_{2}(x)-B_{3}(x)-B_{4}(x)\nonumber\\
                        &= A(x)\Psi(x) + B(x).
  \end{align}
We study the different terms that make up $B(x)$.
  \begin{enumerate}
    \item Recall that the entries of $A(x)$ are either $A_{g, h}(x)=0$ or
      \begin{align*}
        A_{g, s_{\alpha}g}(x) = k_{\alpha} \frac{\langle \alpha, H_{i}\rangle}{\langle \alpha, x\rangle}e^{-i\langle \alpha, x\rangle \langle \alpha, g\lambda\rangle}.
      \end{align*}
      Thus the entries of $\partial_{\xi}A(x)$ are linear combinations of
      \begin{align*}
        \frac{k_{\alpha}}{\langle \alpha, x\rangle^{2}}e^{-i\langle \alpha, x\rangle \langle \alpha, g\lambda\rangle }\qquad \text{and}\qquad \frac{k_{\alpha}}{\langle \alpha, x\rangle} e^{-i\langle \alpha, x\rangle\langle \alpha, g\lambda\rangle}
        \end{align*}
      with $\alpha\in R_{+}\setminus R_{I}$ by the product rule. Using the induction hypothesis (\cref{thm:mainthm}), we see that 
      $\|\omega_{I}^{1/2}(x)\wt\Phi(x)\|\leq C\, |\Omega_{I}^{-1/2}(x)|$. 

      Hence, we may split $B_{1}(x) = D_{1}(x)+D_{2}(x)$, where the entries of $D_{1}(x)$ are bounded by linear combinations of $\frac{k_{\alpha}}{\langle \alpha, x\rangle^{2}}|\Omega^{-1/2}_{I}(x)|$ and the entries of $D_{2}(x)$ are linear combinations of $\frac{k_{\alpha}}{\langle \alpha, x\rangle}e^{-i\langle \alpha, x\rangle \langle \alpha, g\lambda\rangle}\Phi(x)$. 

    \item Since $\partial_{\xi}\omega_{I}^{1/2}(x)=\omega_{I}^{1/2}(x)\sum_{\alpha\in R_{+}\setminus R_{I}}\frac{k_{\alpha}\langle \alpha, \xi\rangle}{\langle \alpha, x\rangle }$, the entries of $B_{2}(x)$ are bounded by linear combinations of $\frac{k_{\alpha}k_{\beta}}{\langle \alpha, x\rangle \langle \beta, x\rangle}|\Omega_{I}^{-1/2}(x)|$ with $\alpha, \beta\in R_{+}\setminus R_{I}$.

    \item By the Leibniz rule, the entries of $B_{3}(x)$ are bounded by linear combinations\\ of $\frac{k_{\alpha}k_{\beta}}{\langle \alpha, x\rangle\langle \beta, x\rangle }|\Omega_{I}^{-1/2}(x)|$ with $\alpha, \beta\in R_{+}\setminus R_{I}$.

    \item Since $R_{I}$ is chosen according to $H_{i}$, we know that
      \begin{align*}
        T_{H} = \partial_{H} + \sum_{\alpha\in R_{+}\setminus R_{I}} k_{\alpha}\langle \alpha, H\rangle \frac{1-s_{\alpha}}{\langle \alpha, \cdot\rangle },
      \end{align*}
      see \cref{eqn:dunklOpParabolicDegen}.
      Thus
      \begin{align*}
        &\partial_{H}(e^{-i\langle x, w\lambda\rangle}\Exp_{k}(iw\lambda, x)) \\
        =&\, e^{-i\langle x, w\lambda\rangle }(\partial_{H}\Exp_{k}(iw\lambda, x)-\langle H, iw\lambda\rangle \Exp_{k}(iw\lambda, x))\\
        =&\, e^{-i\langle x, w\lambda\rangle }(\partial_{H}-T_{H})\Exp_{k}(iw\lambda, x)\\
        =&\, e^{-i\langle x, w\lambda\rangle }\sum_{\alpha\in R_{+}\setminus R_{I}}\frac{k_{\alpha}\langle \alpha, H\rangle}{\langle \alpha, x\rangle }(\Exp_{k}(is_{\alpha}w\lambda, x)-\Exp_{k}(iw\lambda, x)).
      \end{align*}
      From this we conclude that $B_{4}(x)$ is also bounded by linear combinations\\ of $\frac{k_{\alpha}k_{\beta}}{\langle \alpha, x\rangle \langle \beta, x\rangle }|\Omega_{I}^{-1/2}(x)|$ with $\alpha, \beta\in R_{+}\setminus R_{I}$.

  \end{enumerate}
 
  Again by abuse of notation, we abbreviate 
  \begin{align}
     \Psi(t)\coloneqq \Psi(x_{0}+tH_{i}) \qquad (\forall t\geq 0)
  \end{align}
  and similarly $A(t), B(t)$. We then obtain from \cref{eqn:systemPsi} that
  \begin{align}\label{eqn:systemDeriv}
    \frac{d}{dt}\Psi(t) = A(t)\Psi(t) + B(t) \qquad (\forall t\geq 0).
  \end{align}
  The key observation will be that the perturbation $B(t)$ in the system won't affect the bound on $\Psi$ in terms of $|\Omega_{I}^{-1/2}|$.

  To this end, we begin with a technical lemma investigating $B(t)$, the proof of which we postpone to the Appendix (\cref{sec:appendix}).
  \begin{lemma}\label{lem:integrabilityB}
      Let $Q$ be as in \cref{prop:AQentries} and $T\geq 0$ as in \cref{lem:QBound}.
      We then have
      \begin{enumerate}
        \item $\displaystyle \left\|\int_{0}^{t} B(s)\ ds\right\| \leq \text{const.}\cdot |\Omega_{I}^{-1/2}(x_{0})|$ for all $0\leq t\leq T$.
        \item $\displaystyle \left\|\int_{T}^{t} (\mathbbm{1}+Q(s))^{-1} B(s)\ ds\right\| \leq \text{const.}\cdot |\Omega_{I}^{-1/2}(x_{0})|$ for all $t\geq T$.
      \end{enumerate}
    \end{lemma}
  We now proceed as in the proof of \cref{thm:mainthm}. 
  The induction hypothesis now entails
  \begin{lemma}[Good estimate for $\Psi(0)$] 
    There is a constant $C\geq 0$ such that
    \begin{align*}
      \|\Psi(0)\| \leq C\, |\Omega_{I}^{-1/2}(x_{0})|
    \end{align*}
  \end{lemma}
  \begin{proof}
    Recall (\ref{ineq:OmegakOmegaI}), which stated that $|\Omega_{k}^{-1/2}|\leq |\omega_{I}^{-1/2}\Omega_{I}^{-1/2}|$ and the definition (\ref{eqn:defPsi}) of $\Psi$. We employ the Leibniz rule. For the lower order derivatives of $\Exp_{k}(i\lambda,\cdot)$, we already established a bound by $\text{const.}\cdot \Omega_{k}^{-1/2}(x_{0})$ by the induction hypothesis on $\deg(p)$. For the highest order derivative of $\Exp_{k}(i\lambda,\cdot)$ (here $\partial_{\xi}\Exp_{k}(i\lambda, \cdot)$) the induction hypothesis in step $i+1\to i$ states that
    \begin{align*}
      |\partial_{\xi}\Exp_{k}(i\lambda, x_{0})| \leq \text{const.}\cdot |\Omega_{k}^{-1/2}(x_{0})|. 
    \end{align*}
  \end{proof}
  Compare this to \cref{lem:indHypothesisPhi}.
  Thanks to \cref{prop:dunklKernel}(4) and the Leibniz rule, we establish the initial estimate
  \begin{lemma}[Initial estimate for $\Psi(t)$]
    We keep the notation from \cref{lem:initialEst}.
    There is a constant $C\geq 0$ such that
    \begin{align*}
      \|\Psi(t)\|\leq C\, (a+t)^{\mu_{I}}, 
    \end{align*}
  \end{lemma}
  Compare this to \cref{lem:initialEst}. 
  We extend the estimate for $\Psi(0)$ to one for $\Psi(t)$ for $t\in [0, T]$ similar to \cref{lem:indHypothesisPhiExtended}:
  \begin{lemma}[Good estimate for $\Psi(t)$ for $0\leq t\leq T$]\label{lem:indHypExtendedPsi}
     There is a constant $C\geq 0$ such that
    \begin{align*}
      \|\Psi(t)\|\leq C\, |\Omega_{I}^{-1/2}(x_{0})|
    \end{align*}
    for all $t\in [0, T]$.
  \end{lemma}
  \begin{proof} 
    From \cref{eqn:systemDeriv} and \cref{lem:integrabilityB} it follows that for all $0\leq t\leq T$,
  \begin{align}\label{ineq:intFormulaPsi}
    \|\Psi(t)\| &\leq \|\Psi(0)\| + \int_{0}^{t} \|A(s)\|\|\Psi(s)\|\ ds + \left\|\int_{0}^{t} B(s) ds\right\|\nonumber\\
                &\leq \|\Psi(0)\| + \text{const.}\cdot\int_{0}^{t} \frac{(a+s)^{\mu_{I}}}{a}\ ds + \text{const.}\cdot |\Omega_{I}^{-1/2}(x_{0})|\nonumber\\
                &\leq \text{const.}\cdot \left(|\Omega_{I}^{-1/2}(x_{0})| + \frac{T(a+T)^{\mu_{I}}}{a}\right).
  \end{align}
  We continue with a bootstrap argument as in the proof of \cref{lem:indHypothesisPhiExtended}, successively lowering the exponent of $a+T$ in \cref{ineq:intFormulaPsi}. This completes the proof of the lemma.
\end{proof}

    For $t\geq T$, we again consider the transformation $Z(t)=(\mathbbm{1}+Q(t))^{-1}\Psi(t)$ with $Q$ from \cref{prop:AQentries} and calculate
  \begin{align}
    \frac{d}{dt}Z(t) = (\mathbbm{1}+Q(t))^{-1}A(t)Q(t)Z(t)+(\mathbbm{1}+Q(t))^{-1}B(t)
  \end{align}
  as in \cref{eqn:derivationZtransform,eqn:systemMod}.

  \begin{lemma}[Good estimate for $Z(t)$ for $t\geq T$]\label{lem:bootstrapZPsi}
    There is a constant $C\geq 0$ such that 
      \begin{align*}
       \|Z(t)\| \leq C\, |\Omega_{I}^{-1/2}(x_{0})|.
     \end{align*}
  \end{lemma}
  \begin{proof}
  Note that
  \begin{align}
     \|Z(T)\|\leq C\, |\Omega_{I}^{-1/2}(x_{0})|
  \end{align}
  as a consequence of \cref{lem:QBound} and \cref{lem:indHypExtendedPsi}. Since
  \begin{align*}
    Z(t) = Z(T) + \int_{T}^{t} (\mathbbm{1}+Q(s))^{-1}A(s)Q(s)Z(s)\ ds + \int_{T}^{t} (\mathbbm{1}+Q(s))^{-1}B(s)\ ds, 
  \end{align*}
  we calculate (using \cref{lem:integrabilityB})
  \begin{align}
    \|Z(t)\| &\leq \|Z(T)\| + \text{const.}\cdot\int_{0}^{t}(a+s)^{\mu_{I}-2}\ ds +  \left\|\int_{T}^{t} (\mathbbm{1}+Q(s))^{-1}B(s)\ ds\right\|\nonumber\\
             &\leq \text{const.}\cdot \left(|\Omega_{I}^{-1/2}(x_{0})| + \int_{0}^{t} (a+s)^{\mu_{I}-2}\ ds\right).
  \end{align}
  Compare this to (\ref{ineq:Zintegral}).
  We proceed with a bootstrap argument as in \cref{lem:bootstrapZ}, successively lowering exponents, and finally obtain the lemma.
\end{proof}
  
  Thus, by combining \cref{lem:QBound}, \cref{lem:bootstrapZPsi} and \cref{lem:indHypExtendedPsi} we conclude
  \begin{lemma}[Good estimate for $\Psi(t)$ for $t\geq 0$]
    There is a constant $C\geq 0$ such that 
      \begin{align*}
       \|\Psi(t)\| \leq C\, |\Omega_{I}^{-1/2}(x_{0})|.
     \end{align*}
  \end{lemma} 
  Recall \cref{eqn:OmegaX0X}, i.e. $\Omega_{I}(x_{0})=\Omega_{I}(x)$. Hence, considering the first entry of $\Psi$ and putting $t=t(x)$, we conclude
  \begin{align*}
    |\partial_{\xi}(e^{-i\langle x, \lambda\rangle}\Exp_{k}(i\lambda, x))| \leq \text{const.}\cdot |\Omega_{k}^{-1/2}(x)|. 
  \end{align*}
  Since
  \begin{align*}
    &\partial_{\xi}(e^{-i\langle x, \lambda\rangle}\Exp_{k}(i\lambda, x))= e^{-i\langle x, \lambda\rangle}(\partial_{\xi}-T_{\xi})\Exp_{k}(i\lambda, \cdot)(x),
  \end{align*}
  we may estimate
  \begin{align*}
    |\partial_{\xi}\Exp_{k}(i\lambda, x)| &\leq |T_{\xi}\Exp_{k}(i\lambda, x)| + |(\partial_{\xi}-T_{\xi})\Exp_{k}(i\lambda, x)| \\
                                          &\leq \text{const.}\cdot |\Omega_{k}^{-1/2}(x)|.
  \end{align*}
  Here, the induction hypothesis for $\deg(p)=0$ (which corresponds to \cref{thm:mainthm}) and the eigenvalue equation for the Dunkl kernel were used. 

  We thus conclude the induction step $i+1\to i$.

  This proves the claim for $\deg(p)\leq 1$. Using induction on $\deg(p)$, we may now prove the theorem: Due to the Leibniz rule, we can always obtain a system
  \begin{align*}
    \partial_{H}\Psi(x)=A(x)\Psi(x)+B(x) 
  \end{align*}
  for $\Psi(x)=\omega_{I}^{1/2}(x) p(\partial)\wt\Phi(x)$ with perturbation $B$ which we can estimate as in \cref{lem:integrabilityB}, see the proof of \cref{lem:integrabilityB} in \cref{sec:appendix}. This allows for repeating the arguments.

\section{Proof of \cref{thm:mainThmComplex}}\label{sec:proofCplx}
We prove this by induction on the degree $\deg(p)$. The majority of this section will be devoted to proving the claim for $\deg(p)=0$. Indeed, the inductive step then is just an iteration of the arguments in the proof of \cref{thm:mainThmDeriv} and the base case.
For now, we assume that $\lambda=\lambda_{1}+i\lambda_{2}\in \weylcham +i\areg$ and $\deg(p)=0$.
In principle, we again follow the proof of \cref{thm:mainthm}. That is, we establish bounds for $\Exp_{k}(\lambda, \cdot)$ by inductively passing from regular elements in a fixed sector $S$ to more singular ones. Thus, again fix a sector $S$ with subsets $\a^{0}, \a^{n}, \ldots, \a^{1}$ as in \cref{eqn:defai}, let $x\in\a^{i}$ and define $x_{0}=x-t(x)H_{i}\in\a^{0}\cup\a^{i+1}, t(x)\geq 0, a=a(x_{0})\geq 1$ as in \cref{eqn:defx0,eqn:ax0}.  

However, there is one crucial difference. If we define the differential system (\ref{eqn:matrixSystemPhi}) with respect to this $H_{i}\in\weylcham$ as before, we would obtain $\partial_{H_{i}}\Phi = A\Phi$. But the factors $e^{\langle \alpha, x\rangle\langle \alpha, w\lambda\rangle}$ in the entries of $A$ could exponentially grow as $|x|\to\infty$ for $\lambda\in\a_{\C}$, thus destroying any hope for a (conditionally) integrable matrix $A$. 
  Hence, we shall now instead redefine
  \begin{align}\label{eqn:systemPhiRedef}
    \Phi(x) \coloneqq \Phi^{H_{i},\lambda}(x) \coloneqq (\omega_{I}^{1/2}(x)e^{-i\langle x, w\lambda_{2}\rangle}\Exp_{k}(w\lambda, x))_{w\in W}
  \end{align}
  and obtain
  \begin{align}
    \partial_{H_{i}} \Phi(x) = \Gamma \Phi(x) + A(x)\Phi(x)\quad\text{with}\quad \Gamma=\diag((\langle H_{i}, w\lambda_{1}\rangle)_{w\in W})
  \end{align}
  and
  \begin{align*}
    A^{H_{i},\lambda}_{g, h}(x) \coloneqq\begin{cases}
      k_{\alpha}\frac{\langle \alpha, H_{i}\rangle }{\langle \alpha, x\rangle }e^{-i\langle \alpha, x\rangle \langle \alpha, g\lambda_{2}\rangle} &\text{if}\ h=s_{\alpha}g,\, \alpha\in R_{+}\setminus R_{I},\\
      0 &\text{otherwise.}
    \end{cases}
  \end{align*}
  We again put $\Phi(t)\coloneqq \Phi(x_{0}+tH_{i})$, with the notations as before, and obtain
  \begin{align}\label{eqn:systemComplex}
  \frac{d}{dt} \Phi(t) = (\Gamma+A(t))\Phi(t) \qquad(\forall t\geq 0).
  \end{align}
  In this way, we again read the system $\frac{d}{dt}\Phi = (\Gamma+ A)\Phi$ as a perturbation of the system $\frac{d}{dt}Y=\Gamma Y$ with a conditionally integrable perturbation $A$. Note that (\ref{eqn:matrixSystemPhi}) in this formulation has $\Gamma = 0$. The system thus becomes asymptotically constant, but not necessarily asymptotically zero as before. The diagonal matrix $\Gamma$, which has only real entries, will thus be responsible for the exponential terms in the bounds of \cref{thm:mainThmComplex}.

  Note that we could also arrive at this system by denoting the system (\ref{eqn:systemHi}) by $\frac{d}{dt}\wt\Phi(t)=\wt A(t)\wt\Phi(t)$ and considering 
  \begin{align*}
    \Phi(t)\coloneqq e^{\Gamma_{0}+t\Gamma} \wt\Phi(t)\qquad\text{and}\qquad A(t)=e^{\Gamma_{0}+t\Gamma}\wt A(t)e^{-\Gamma_{0}-t\Gamma}
  \end{align*}
  with $\Gamma_{0}=\diag(\langle x_{0}, w\lambda_{1}\rangle)_{w\in W}$.
 This will be relevant later.

  Previously, we used the transformation $Z = (1+Q)^{-1}\Phi$ with $Q(t)=-\int_{t}^{\infty} A(s)\, ds$\ \ to obtain the system (\ref{eqn:systemMod}) from the system (\ref{eqn:system}).
  This was a decisive step in the proof, as it transforms a system with a conditionally integrable coefficient matrix into a system with a coefficient matrix that belongs to $L^{1}([T, \infty[, \C^{|W|\times |W|})$. Indeed, recall \cref{lem:estNonParabRootsBetter}, which showed that the entries of the coefficient matrix in (\ref{eqn:systemMod}) are $\mathcal{O}((a+t)^{-2})$ whereas those of the coefficient matrix in (\ref{eqn:system}) are only $\mathcal{O}((a+t)^{-1})$. Here, $\mathcal O$ denotes the usual $\mathcal O$-notation, and $a=a(x_{0})\geq 1$ is that of the induction step defined in (\ref{eqn:ax0}).

  This was crucial for the bootstrap argument in \cref{lem:bootstrapZ}, where we successively sharpened the initial estimate on $Z(t)$ from \cref{lem:initialEstMod} until we reached the desired bound. This then led to a bound for $\Phi$ in \cref{lem:bootstrapPhiFinish}.
 
  Performing the same transformation $Z=(\mathbbm{1}+Q)^{-1}\Phi$ with the same $Q(t)=-\int_{t}^{\infty}A(s)\, ds$\ \ on the system (\ref{eqn:systemComplex}), however, does not result in an improvement of the integrability, as the diagonal matrix $\Gamma$ is not necessarily zero. In fact, following the calculations of \cite[Eq. (1.11.6)]{eastham} like we did in (\ref{eqn:derivationZtransform}), now with a non-zero diagonal $\Gamma$, leads us to
  \begin{align*}
    Z' = (\Gamma + (\mathbbm{1}+Q)^{-1}[\Gamma, Q] + S)Z 
  \end{align*}
  with $S=(\mathbbm{1}+Q)^{-1}AQ\in L^{1}([T, \infty[, \C^{|W|\times |W|})$. But the term $[\Gamma, Q]$ possibly has the same poor conditional integrability as $A$ does, as its entries are certain linear combinations of trigonometric integrals, see \cref{prop:AQentries}(2).

  Therefore, we need to apply a different modification to (\ref{eqn:systemComplex}) before we can repeat the bootstrap argument of \cref{lem:bootstrapZ}.
  To this end, we still use the transformation $Z=(\mathbbm{1}+Q)^{-1}\Phi$, but with a different $Q$. 
  \begin{lemma}\label{lem:constructionQcomplex}
    There exists a matrix-valued function $Q\colon [0, \infty[\to\C^{|W|\times |W|}$ such that 
    \begin{align*}
      Q' = [\Gamma, Q]+A\quad \text{and}\quad Q(t)\in \mathcal{O}((a+t)^{-1}).
    \end{align*}
    The constants absorbed in the $\mathcal O$-notation do not depend on $a=a(x_{0})\geq 1$. In particular, there exists a constant $T\geq 0$, independent of $a(x_{0})$, such that $\mathbbm{1}+Q(t)$ is invertible for $t\geq T$ and the $(\mathbbm{1}+Q(t))^{-1}$ are uniformly bounded.
  \end{lemma}
  The construction of this $Q$ works as in the proof of \cite[Thm. 4.26]{bodineLutz}, which provides a different version of the Levinson-type theorem \cite[Thm. 1.11.1]{eastham} we used earlier. For the reader's convenience, a proof of \cref{lem:constructionQcomplex}, which reiterates this construction, is given in the Appendix (\cref{sec:appendix}).
    One should compare the properties of $Q$ from \cref{lem:constructionQcomplex} with those of the function $Q$ in \cref{prop:AQentries}, \cref{lem:estNonParabRootsBetter} and \cref{lem:QBound}. While the entries of $AQ$ are now no longer as explicit, we still know they are $\mathcal O((a+t)^{-2})$, which suffices for our purposes.

  Indeed, using this new $Q$, we again differentiate $\Phi=(\mathbbm{1}+Q)Z$ and use the system (\ref{eqn:systemComplex}) to see
  \begin{align}
    (\mathbbm{1}+Q)Z' + ([\Gamma, Q] + A)Z = \Phi' = (\Gamma + A)(\mathbbm{1}+Q)Z.
  \end{align}
  Rearranging terms, this leaves us with
  \begin{align}\label{eqn:systemComplexMod}
    Z'(t) = \Big(\Gamma + \big(\mathbbm{1}+Q(t)\big)^{-1}A(t)Q(t)\Big) Z(t) \qquad(\forall t\geq T).
  \end{align}
  Thus 
  \begin{align}\label{eqn:systemComplexModShort}
     Z'(t)=\big(\Gamma+S(t)\big)Z(t)\qquad (\forall t\geq T)
  \end{align}
  with $S\in \mathcal O((a+t)^{-2})$. In particular, $S\in L^{1}([T, \infty[, \C^{|W|\times |W|})$.
  Now we can finally adapt the bootstrap arguments from the proof of \cref{lem:bootstrapZ}, which are then quite similar to those of \cite[Sec. 7]{clercBessel}.

  Note that $\langle y, w\lambda_{1}\rangle \leq \langle y, \lambda_{1}\rangle $ for all $y\in\weylcham, w\in W$, as $\lambda_{1}\in\weylcham$.
  We therefore have the following initial estimate as a direct consequence of \cref{eqn:systemPhiRedef} and \cref{prop:dunklKernel}(4).
  \begin{lemma}[Initial estimate for $\Phi(t)$]\label{lem:initialEstComplex}
    We keep the notation of \cref{lem:initialEst}. There exists a constant $C\geq 0$ such that 
    \begin{align}\label{ineq:initialEstComplex}
    \|\Phi(t)\|\leq C\, e^{t\langle H_{i}, \lambda_{1}\rangle}e^{\langle x_{0}, \lambda_{1}\rangle}(a+t)^{\mu_{I}} 
  \end{align}
  \end{lemma}

  The induction hypothesis $|\Exp_{k}(\lambda, x_{0})|\leq \text{const.}\cdot e^{\langle x_{0}, \lambda_{1}\rangle }|\Omega_{k}^{-1/2}(x_{0})|$ now entails
  \begin{lemma}[Good estimate for $\Phi(0)$]
    There is a constant $C\geq 0$ such that
    \begin{align}\label{eqn:indHypComplex}
      \|\Phi(0)\| \leq C\, e^{\langle x_{0}, \lambda_{1}\rangle}|\Omega_{I}^{-1/2}(x_{0})|. 
    \end{align}
  \end{lemma}  
  Again, the operators $\mathbbm{1}+Q(t)$ become invertible only for $t\geq T$, thus we first have to extend the induction hypothesis from an estimate for $\|\Phi(0)\|$ to one for $\sup_{t\in [0, T]}\|\Phi(t)\|$.

  \begin{lemma}[Good estimate for $\Phi(t)$ for $0\leq t\leq T$]\label{lem:indHypComplexExtended}
    There is a constant $C\geq 0$ such that
    \begin{align*}
      \|\Phi(t)\| \leq C\, e^{\langle x_{0}+tH_{i}, \lambda_{1}\rangle}|\Omega_{I}^{-1/2}(x_{0})|
    \end{align*}
    for all $t\in [0, T]$.
  \end{lemma}
  \begin{proof}

    Since $\Phi'(t)=\Gamma\Phi(t)+A(t)\Phi(t)$, and using the initial estimate from \cref{lem:initialEstComplex} and \cref{lem:estNonParabRootsBetter}(1), we in particular obtain
  \begin{align}
    \Phi'(t) = \Gamma\Phi(t) + b(t) 
  \end{align}
  with 
  \begin{align}
     \|b(t)\|\leq Ca^{-1}e^{\langle x_{0}+tH_{i}, \lambda_{1}\rangle}(a+t)^{\mu_{I}}.
  \end{align}
   Thus, using the well-known solution form of constant coefficient linear ODEs 
   \begin{align}\label{eqn:intFormPhiComplex}
    \Phi(t) = e^{t\Gamma}\Phi(0) + \int_{0}^{t}e^{(t-s)\Gamma}b(s)\ ds 
  \end{align}
  and the fact $\|e^{t\Gamma}\|\leq \text{const.}\cdot e^{t\langle H_{i}, \lambda_{1}\rangle}$, we obtain
  \begin{align}\label{ineq:phiComplex}
    \|\Phi(t)\| &\leq \text{const.}\cdot e^{t\langle H_{i}, \lambda_{1}\rangle }\left(\|\Phi(0)\| + \int_{0}^{t}e^{-s\langle H_{i}, \lambda_{1}\rangle } \|b(s)\|\ ds\right)\\
                &\leq \text{const.}\cdot e^{t\langle H_{i}, \lambda_{1}\rangle}e^{\langle x_{0}, \lambda_{1}\rangle}\left(|\Omega_{I}^{-1/2}(x_{0})|+\int_{0}^{t} a^{-1}(a+s)^{\mu_{I}}\ ds\right)\nonumber\\
                &\leq \text{const.} \cdot e^{\langle x_{0}+tH_{i}, \lambda_{1}\rangle } (|\Omega_{I}^{-1/2}(x_{0})| + a^{\mu_{I}-1})
  \end{align}
  for all $t\leq T$ (again $T$ is independent of $a=a(x_{0})\geq 1$ and treated as a constant). We successively shrink the exponent as in the proof of \cref{lem:indHypothesisPhiExtended} by plugging increasingly better estimates for $\|b(s)\|$ into (\ref{ineq:phiComplex}) to eventually obtain
  \begin{align*}
    \|\Phi(t)\| \leq \text{const.}\cdot e^{\langle x_{0}+tH_{i}, \lambda_{1}\rangle }|\Omega_{I}^{-1/2}(x_{0})|
  \end{align*}
  for all $t\in [0, T]$.
  \end{proof}

  To handle $t\geq T$, we now employ the system (\ref{eqn:systemComplexModShort}). 
  \begin{lemma}[Good estimate for $Z(t)$ for $t\geq T$]\label{lem:bootstrapZComplex}
    There is a constant $C\geq 0$ such that
    \begin{align*}
      \|Z(t)\|\leq C\, e^{\langle x_{0}+tH_{i},\lambda_{1}\rangle}|\Omega_{I}^{-1/2}(x_{0})| 
    \end{align*}
    for all $t\geq T$.
  \end{lemma}
  \begin{proof}
    Recall \cref{eqn:systemComplexModShort}, which stated $Z'(t)=(\Gamma + S(t))Z(t)$ with $S(t)\in\mathcal O((a+t)^{-2})$. From this and the initial estimate in \cref{lem:initialEstComplex}, we conclude that
  \begin{align*}
    Z'(t) = \Gamma Z(t) + b(t) 
  \end{align*}
  for all $t\geq T$ with some $b(t)$ satisfying
  \begin{align}\label{ineq:estBComplexMod}
     \|b(t)\|\leq C e^{\langle x_{0}+tH_{i}, \lambda_{1}\rangle }(a+t)^{\mu_{I}-2}.
  \end{align}
  Now
  \begin{align}\label{eqn:intFormulaZMod}
    Z(t) = e^{(t-T)\Gamma}Z(T) + \int_{T}^{t}e^{(t-s)\Gamma}b(s)\ ds,
  \end{align}
  thus
  \begin{align}\label{ineq:intFormulaZMod}
    \|Z(t)\| &\leq e^{t\langle H_{i}, \lambda_{1}\rangle} \left(\|Z(T)\| + \int_{T}^{t} e^{-s\langle H_{i}, \lambda_{1}\rangle}\|b(s)\|\ ds\right)\\
             &\leq \text{const.}\cdot e^{t\langle H_{i},\lambda_{1}\rangle}\left(e^{\langle x_{0}, \lambda_{1}\rangle}|\Omega_{I}^{-1/2}(x_{0})|+\int_{0}^{t}e^{\langle x_{0}, \lambda_{1}\rangle} (a+s)^{\mu_{I}-2}\ ds\right)\nonumber\\
             &\leq \text{const.}\cdot e^{\langle x_{0}+tH_{i}, \lambda_{1}\rangle}\left(|\Omega_{I}^{-1/2}(x_{0})|+\int_{0}^{t}(a+s)^{\mu_{I}-2}\ ds\right).
  \end{align}
  Compare this to (\ref{ineq:Zintegral}). Proceeding with the arguments as in the proof of \cref{lem:bootstrapZ}, we can thus obtain
  \begin{align}
    \|Z(t)\| &\leq \text{const.}\cdot e^{t\langle H_{i}, \lambda_{1}\rangle}e^{\langle x_{0}, \lambda_{1}\rangle }|\Omega_{I}^{-1/2}(x_{0})|\nonumber\\
             &=\text{const.}\cdot e^{\langle x_{0}+tH_{i}, \lambda_{1}\rangle }|\Omega_{I}^{-1/2}(x_{0}+tH_{i})|
  \end{align}
  for all $t\geq T$. 
  \end{proof}

  Again, by uniform boundedness (see \cref{lem:constructionQcomplex}), we conclude that
  \begin{lemma}[Good estimate for $\Phi(t)$ for $t\geq 0$]
     There exists a constant $C\geq 0$ such that
     \begin{align*}
       \|\Phi(t)\|\leq C\, e^{\langle x_{0}+tH_{i}, \lambda_{1}\rangle}|\Omega_{I}^{-1/2}(x_{0})| 
     \end{align*}
     for all $t\geq 0$.
  \end{lemma}
  This furnishes the proof for $\deg(p)=0$. Indeed, we now consider all entries of $\Phi$, not just the first one, to see
  \begin{align*}
    |\Exp_{k}(w\lambda, x)| \leq C e^{\langle \lambda_{1}, x\rangle}|\Omega_{k}^{-1/2}(x)|
  \end{align*}
  for $w\in W$ and $x\in\weylcham$. Using \cref{prop:dunklKernel}(2), we obtain the desired result for $x\in\a$ and thus for $\lambda=\lambda_{1}+i\lambda_{2}\in\a+i\areg$.

  To finally prove the theorem for arbitrary $p\in \C[\a]$, we again use induction on the degree $\deg(p)$ as in \cref{thm:mainThmDeriv}. Indeed, we start out with the system (\ref{eqn:systemDeriv}), which we rename
  \begin{align*}
    \wt\Psi'(t) = \wt A(t)\wt\Psi(t) + \wt B(t).
  \end{align*}
  Note that we now take the spectral parameter $\lambda=\lambda_{1}+i\lambda_{2}\in\a+i\areg$, whereas in (\ref{eqn:systemDeriv}) we used $i\lambda\in i\areg$ as the spectral parameter.
  We then define 
  \begin{align*}
    \Psi(t) &\coloneqq e^{\Gamma_{0}+t\Gamma}\wt\Psi(t) \\
            &= \Big(e^{\langle x_{0}+tH_{i}, w\lambda_{1}\rangle}\omega_{I}^{1/2}(x_{0}+tH_{i})p(\partial)(e^{-\langle x_{0}+tH_{i}, w\lambda\rangle }\Exp_{k}(w\lambda, x_{0}+tH_{i}))\Big)_{w\in W},
  \end{align*}
  where $\Gamma_{0}=\diag(\langle x_{0}, w\lambda_{1}\rangle)_{w\in W}$. Next, we put
  \begin{align*}
    A(t)&=e^{\Gamma_{0}+t\Gamma}\wt A(t)e^{-\Gamma_{0}-t\Gamma}\qquad\text{and} \qquad B(t)=e^{\Gamma_{0}+t\Gamma}\wt B(t).
  \end{align*}
  Note that this $A$ is the same $A$ as in (\ref{eqn:systemComplex}). Indeed, the matrix $\wt A$ is exactly the coefficient matrix of the system (\ref{eqn:matrixSystemPhi}). Thus, $\wt A$ is also the coefficient matrix of the system (\ref{eqn:systemDeriv}). Conjugation by $e^{\Gamma_{0}+t\Gamma}$ removes the diagonal terms of $\Gamma$ in the system.

  We thus obtain
  \begin{align}
    \frac{d}{dt}\Psi(t) = (\Gamma+A(t))\Psi(t) + B(t) 
  \end{align}
  with a rest term $B$ that satisfies similar properties as in \cref{lem:integrabilityB}. Indeed, we have
  \begin{lemma}\label{lem:integrabilityBcomplex}
    Let $Q$ and $T\geq 0$ be as in \cref{lem:constructionQcomplex}. We then have
    \begin{enumerate}
      \item $\displaystyle \left\|\int_{0}^{t} e^{(t-s)\Gamma}B(s)\ ds\right\| \leq \text{const.}\cdot e^{\langle x_{0}+tH_{i}, \lambda_{1}\rangle}|\Omega_{I}^{-1/2}(x_{0})|$ for all $0\leq t\leq T$.
      \item $\displaystyle \left\|\int_{T}^{t} e^{(t-s)\Gamma}(\mathbbm{1}+Q(s))^{-1} B(s)\ ds\right\| \leq \text{const.}\cdot e^{\langle x_{0}+tH_{i}, \lambda_{1}\rangle}|\Omega_{I}^{-1/2}(x_{0})|$ for all $t\geq T$.
    \end{enumerate}
  \end{lemma}
  We can now proceed as before to finally conclude the proof. Indeed, note that the transformation $(\mathbbm{1}+Q)Z=\Psi$ with $Q$ from \cref{lem:constructionQcomplex} yields
  \begin{align*}
    (\mathbbm{1}+Q)Z' + ([\Gamma, Q]+A)Z = \Psi' = (\Gamma+A)(\mathbbm{1}+Q)Z+B.
  \end{align*}
  Hence,
  \begin{align}
    Z' = \Gamma Z + (\mathbbm{1}+Q)^{-1}AQ Z + (\mathbbm{1}+Q)^{-1}B,
  \end{align}
  and thus
  \begin{align}
    Z'(t) = \Gamma Z(t) + b(t) + (\mathbbm{1}+Q(t))^{-1}B(t)
  \end{align}
  The rest term $b(t)$ is bounded by $Ce^{\langle x_{0}+tH_{i}, \lambda_{1}\rangle}(a+t)^{\mu_{I}-2}$ as a consequence of \cref{lem:initialEstComplex}, $A(t)Q(t)\in\mathcal O((a+t)^{-2})$ and \cref{lem:constructionQcomplex}. 
  We now repeat the arguments from \cref{lem:indHypComplexExtended} and \cref{lem:bootstrapZComplex} and use the estimates from \cref{lem:integrabilityBcomplex}.

\section{Applications}\label{sec:applications}
In this section, we give some applications of our estimates.

\begin{theorem}\label{prop:lp}
  Let $\lambda\in\areg$, $p>0$ and $\Re\,k>\frac 1 p$. Then $\Exp_{k}(i\lambda, \cdot)\in L^{p}(\a)$. 
\end{theorem}
\begin{proof}
  Denote by $\Delta_{+}=\{\alpha_{1},\ldots,\alpha_{n}\}$ the set of simple roots associated with $\weylcham$ and abbreviate $k_{i}\coloneqq k_{\alpha_{i}}$.
Using the transform $x=\sum_{i=1}^{n}t_{i}\lambda_{i}$ with the dual basis $\langle \alpha_{i}, \lambda_{j}\rangle = \delta_{ij}$, we calculate
  \begin{align*}
    \int_{\a}|\Omega_{k}^{-p/2}(x)| dx &\leq \int_{\a} \prod_{i=1}^{n} (1+|\langle \alpha_{i}, x\rangle|)^{-p\Re\,k_{i}}\ dx\\
                                       &= \text{const.}\cdot \prod_{i=1}^{n} \int_{\R} (1+|t_{i}|)^{-p\Re\,k_{i}}\ dt_{i} < \infty.
  \end{align*}
  The claim now follows from \cref{thm:mainthm}.
\end{proof}

Recall that for multiplicities $k\geq 0$ and $\lambda\in\a$, the Dunkl kernel has a positive representing measure (\cref{prop:repMeasure}), i.e.
\begin{align*}
  \Exp_{k}(-i\lambda, x) = \Exp_{k}(\lambda, -ix)= \int_{\a} e^{-i\langle \xi, x\rangle}\ d\mu_{\lambda}^{k}(\xi)\qquad (\forall x\in\a_{\C})
\end{align*}
with a unique compactly supported probability measure $\mu_{\lambda}^{k}$.
Thus $\Exp_{k}(-i\lambda,\cdot)$ is just the Fourier-Stieltjes transform of the measure $\mu_{\lambda}^{k}$.

\begin{theorem}\label{prop:absCont}
  Suppose $k>\frac 1 2$ and $\lambda\in\areg$. Then $\mu_{\lambda}^{k}$ is absolutely continuous with respect to the Lebesgue measure on $\a$ with density function belonging to $L^{2}(\a)$.
\end{theorem}
\begin{proof}
  Indeed, \cref{prop:lp} implies that $\Exp_{k}(i\lambda, \cdot)$ belongs to $L^{2}(\a)$, thus its Fourier transform $f_{\lambda}^{k}$ is in $L^{2}(\a)$ by the Plancherel theorem. Hence, as distributions, $\mu_{\lambda}^{k}=f_{\lambda}^{k}$ by the inversion theorem for tempered distributions. Thus $\mu_{\lambda}^{k}(g)=\int_{\a}f_{\lambda}^{k}(x) g(x)\, dx$ for all $g\in C^{\infty}_{c}(\a)$. As $f_{\lambda}^{k}\in L^{2}(\a)\subseteq L^{1}_{loc}(\a)$, we may use the density of $C_{c}^{\infty}(\a)$ in $C_{c}(\a)$ to conclude that $\mu_{\lambda}^{k}=f_{\lambda}^{k}\ dx$ as Radon measures, see \cite[Ch. 24]{treves}.
\end{proof}

If $k>1$, then $\Exp_{k}(i\lambda, \cdot)$ belongs to $L^{1}(\a)$. Therefore, its Fourier transform $f_{\lambda}^{k}$ is in $C_{0}(\a)$ by the Riemann-Lebesgue lemma. Following the arguments in the proof of \cref{prop:absCont}, we hence even obtain a continuous density. In fact, we have the following

\begin{proposition}
  Suppose $\ell\in \N_{0}, \lambda\in\areg$ and $k>\ell+\frac{n+1}{2}$. Then the density $f_{\lambda}^{k}$ of $\mu_{\lambda}^{k}$ is of class $C_{c}^{\ell}(\a)$. 
\end{proposition}
\begin{proof}
  Again, denote by $\Delta_{+}=\{\alpha_{1},\ldots, \alpha_{n}\}$ the set of simple roots associated with $\weylcham$ and abbreviate $k_{i}\coloneqq k_{\alpha_{i}}$. Since all norms on $\a$ are equivalent, there exists $c\geq 0$ such that
  \begin{align*}
    1+|x| \ \leq\ 1+c\sum_{i=1}^{n}|\langle \alpha_{i}, x\rangle|\ \leq\ (1+c)\Big(1+\sum_{i=1}^{n}|\langle \alpha_{i}, x\rangle |\Big)\ \leq\ (1+c)\prod_{i=1}^{n}(1+|\langle \alpha_{i}, x\rangle|)
  \end{align*}
  for all $x\in\a$. We hence obtain
  \begin{align*}
    \int_{\a} (1+|x|^{2})^{s}\,\Omega_{k}^{-1}(x)\ dx &\leq \int_{\a} \frac{(1+|x|)^{2s}}{\prod_{i=1}^{n}(1+|\langle \alpha_{i}, x\rangle|)^{2k_{i}}}\ dx\\
                                                      &\leq \text{const.}\cdot\int_{\a} \prod_{i=1}^{n} \frac{(1+|\langle \alpha_{i}, x\rangle|)^{2s}}{(1+|\langle \alpha_{i}, x\rangle|)^{2k_{i}}}\ dx\\
                                                      &\leq \text{const.}\cdot \prod_{i=1}^{n} \int_{\R} (1+|t_{i}|)^{2(s - k_{i})}\ dt_{i}.
  \end{align*}
  The last integral is finite iff $2(s-k_{i})<-1$ for all $i=1,\ldots, n$. Thus, if $s< k_{i}-\frac 1 2$ for all $i=1, \ldots, n$, then $\mu_{\lambda}^{k}$ belongs to the Sobolev space 
  \begin{align*}
    H^{s}(\a) \coloneqq \left\{f\in \S'(\a): \widehat f\in L^{2}_{loc}(\a), \int_{\a} (1+|x|^{2})^{s}|\widehat f(x)|^{2}\ dx < \infty\right\}. 
  \end{align*}
  By the Sobolev embedding theorem (\cite[Ch. 31]{treves}), $H^{s}(\a)\subseteq C^{\ell}(\a)$ for $s>\ell+\frac{n}{2}$. Thus, if there exists $s>0$ such that
  \begin{align*}
    \ell+\frac{n}{2} < s < k_{i}-\frac 1 2, 
  \end{align*}
  we can conclude $\mu_{\lambda}^{k}=f_{\lambda}^{k}\in C^{\ell}(\a)$ as a tempered distribution and thus as a measure by the same argument as in \cref{prop:absCont}. But this is true by our assumption $k>\ell+\frac{n+1}{2}$.
\end{proof}

The calculations in the previous proofs depended on $\Re\,k$ rather than on $k$. Thus we in particular obtain

\begin{corollary}
  Suppose $\ell\in \N_{0}, \lambda\in\areg$ and $\Re\,k>\ell+\frac{n+1}{2}$. Then the Fourier transform $\widehat\Exp_{k}(i\lambda, \cdot)$ is of class $C_{c}^{\ell}(\a)$. 
\end{corollary}
\begin{proof}
  The (classical) Paley-Wiener theorem (\cite[Ch. 29]{treves}) shows that the Fourier transform $\widehat\Exp_{k}(-i\lambda, \cdot)$ is a compactly supported distribution. Now we use the same arguments as in the proof of the previous proposition.
\end{proof}

In fact, the support of $\widehat\Exp_{k}(i\lambda, \cdot)$ is contained in the convex hull $\conv (W.\lambda)$ of the $W$-orbit of $\lambda$, as was already noted in \cite[Cor. 3.3]{dejeu}.

\section{Further remarks}\label{sec:remarks}
We conclude with some final remarks on the geometric situations. To make notation less cumbersome, we shall write $\partial_{p}$ and $T_{p}$ instead of $p(\partial)$ and $p(T)$ for $p\in\C[\a]$.

\subsection{Cartan motion groups}
We restate a few facts which are important for our purposes and refer the reader to \cite[Section 6]{deJeuPaleyWiener} for a detailed discussion. We take a connected non-compact semisimple Lie group $G$ with finite center, maximal compact subgroup $K\subseteq G$, and corresponding Cartan decomposition $\g=\Lie(G)=\k\oplus\p$, where $\k = \Lie(K)$. The Cartan motion group $G_{0}=K\ltimes \p$ acts on the flat symmetric space $G_{0}/K \cong \p$ by Killing form isometries. Take a maximal abelian subspace $\a\subseteq\p$ with restricted root system $\Sigma$, Weyl group $W$, and multiplicities $m_{\alpha}, \alpha\in\Sigma$. Rescaling the root multiplicities according to $k_{\alpha}\coloneqq \frac 1 4 \sum_{\beta\in\R\alpha\cap \Sigma}m_{\beta}$ and choosing a normalized root system $R$ with the same Weyl group as $\Sigma$, the spherical functions $\psi_{\lambda}\ (\lambda\in\a_{\C}/W)$  of $G_{0}/K$ are precisely the Bessel functions, i.e.
\begin{align*}
  \psi_{\lambda}|_{\a} = J_{k}(\lambda, \cdot).
\end{align*}
This is due to the fact that the $W$-invariant Dunkl operators associated with $(R, k)$ appear precisely as the radial parts of $K$-invariant differential operators on $\p$, whose joint eigenfunctions are the spherical functions. 
We hence reobtain the results of \cite{clercBessel} as a consequence of our results by averaging the Dunkl kernel.

\subsection{A Bessel variant of (\ref{eqn:matrixSystemPhi})}
Instead of starting with the Dunkl kernel system \cref{eqn:jointEVProb}, a more classically inspired approach would be to investigate the Bessel system \cref{eqn:jointEVProbBessel}. Thus, consider $H\in\weylcham$ and the corresponding parabolic subgroup $W_{I}\coloneqq \{w\in W : wH=H\}\subseteq W$.
By Chevalley's theorem (\cite{chevalley, hcSphericalI}) we have $\C[\a]^{W_{I}}=\bigoplus_{i=1}^{r} p_{i}\cdot \C[\a]^{W}$ for certain homogeneous $p_{1}=1, \ldots, p_{r}\in \C[\a]^{W_{I}}$, $r=[W:W_{I}]$. Thus, we can find $q_{ij}\in \C[\a]^{W}$ such that
  \begin{align*}
    T_{H}T_{p_{i}} = \sum_{j=1}^{r} T_{p_{j}}T_{q_{ij}}. 
  \end{align*}
  Indeed, note that $H\in \C[\a]^{W_{I}}$ under the usual identification $\a\leftrightarrow \a^{\ast}$.
  The Bessel function $J_{k}(\lambda, \cdot)$ is $W$-invariant and satisfies $$T_{q}J_{k}(\lambda, \cdot)=q(\lambda)J_{k}(\lambda, \cdot)$$ for all $q\in \C[\a]^{W}$.  We hence obtain
  \begin{align*}
    T_{H} \wt\Phi = \Gamma \wt\Phi \quad\text{with}\quad \wt\Phi=\matr{T_{p_{1}}J_{k}(i\lambda, \cdot)\\ \vdots\\ T_{p_{r}}J_{k}(i\lambda, \cdot)}\quad\text{and}\quad\Gamma = \matr{q_{11}(i\lambda) & \ldots & q_{1r}(i\lambda)\\ \vdots & \ddots &\vdots \\ q_{r1}(i\lambda) & \cdots & q_{rr}(i\lambda)}.
  \end{align*}

  Introducing $\Phi(x)\coloneqq \omega_{I}^{1/2}(x)\wt\Phi(x)$, we would obtain
  \begin{align}\label{eqn:besselMatrixSystem}
    \partial_{H}\Phi(x) = \Gamma\Phi(x) + \sum_{\alpha\in R_{+}\setminus R_{I}}\frac{k_{\alpha}\langle \alpha, H\rangle }{\langle \alpha, x\rangle} \omega_{I}^{1/2}(x) \Phi(s_{\alpha}x).
  \end{align}

  This approach would be closer to the spirit of Harish-Chandra's discussion of spherical functions on symmetric spaces of non-compact type (\cite{hcSpherical}; see later), and thus in the spirit of Clerc, who adapted these ideas to conquer the Euclidean case (\cite{clercBessel}).

  Again, the decisive next step in the proof would be to modify this system so that it becomes a perturbation of $\partial_{H}Y=\Gamma Y$ with an absolutely integrable perturbation. 
  In the geometric cases, one basically deals with radial parts of differential operators on certain subalgebras $\p\supseteq\p_{0}\supseteq\a$. Loosely speaking, conjugating them by the weight function $\omega_{I}^{1/2}$ makes them self-adjoint and transforms the system into the desired form, as this eliminates certain coefficients in the operators of poor integrability.  This is the approach of \cite{clercBessel}. 
  However, this approach heavily uses the geometry of the ambient space $\p$ and the explicit realization of the operators as differential operators. Thus, this method fails to generalize to the Dunkl setting. In fact, one would need precise knowledge of how higher-order invariant Dunkl operators degenerate when restricted to a set $\anreg = \{x\in\a : \langle \alpha, x\rangle \neq 0\ (\forall \alpha\in R\setminus R_{I})\}\subseteq\a$ to make such an approach feasible. 
 
  Furthermore, the realization as differential operators in the geometric cases makes it easier to establish first-order systems for the derivatives of spherical functions, as one can simply use the Leibniz rule without paying attention to reflection terms. Thus, incorporating an induction on the degree $\deg(p)$ of a polynomial $p\in\C[\p]$ when considering $\partial_{p}\psi_{\lambda}$ is easier in this situation.

  Obtaining estimates on derivatives of Bessel functions would in turn lead to estimates for the Dunkl kernel, as $\Exp_{k}(\lambda, \cdot) = \partial_{q} J_{k}(\lambda, \cdot)$ for some $q\in \C[\a]$, see \cite{opdamBessel}.
 
  While the system (\ref{eqn:besselMatrixSystem}) remains rather elusive, the Dunkl system (\ref{eqn:matrixSystemPhi}) is very explicit and thus enables us to obtain integrable perturbations by completely different methods (\cite{eastham, bodineLutz}). The spectral properties of $\Gamma$ in (\ref{eqn:besselMatrixSystem}) were already studied by Harish-Chandra (\cite{hcSpherical}). Note that the $T_{p_{j}}J_{k}(i\lambda,\cdot)$ are $W_{I}$-invariant linear combinations of $\Exp_{k}(iw\lambda, \cdot), w\in W$ by the eigenvalue property of the Dunkl kernel. Thus by studying (\ref{eqn:matrixSystemPhi}) we effectively diagonalize the system (\ref{eqn:besselMatrixSystem}).
  
  Thus, the situation can be summarized as follows: In the general Dunkl setting, we lack important tools used by Clerc to make his differential system sufficiently nice to study the asymptotic properties of its solutions. However, the first-order Dunkl operators are much more explicit than the radial parts of (higher-order) invariant differential operators. In particular, the explicit degeneration of $T_{H}$ when $H\in\weylcham$ becomes singular, together with the eigenvalue property of the Dunkl kernel, allows us to obtain first-order systems where all entries are precisely known. This allows us to employ classical tools of Eastham \cite{eastham, bodineLutz} to study the asymptotics of its solutions, which would not have worked with Clerc's system. Thus, our approach also clarifies the geometric situation.

\subsection{Harish-Chandra's theory and Heckman-Opdam hypergeometric functions}
The idea of studying spherical functions using radial parts and first-order systems dates back to Harish-Chandra (\cite{hcDiffOps, hcSphericalI, hcSpherical}), where he used this to investigate spherical functions of Riemannian symmetric spaces of the non-compact type $G/K$. 
The striking observation was that the radial part of the Casimir operator, and thus the radial part of the Laplace-Beltrami operator, plays a predominant role. The corresponding differential equation has regular singularities, and thus classical Frobenius approaches for Fuchsian differential equations generalize to this situation (\cite{hcSpherical,trombiVaradarajan, casselmanMilicic, gangolliVaradarajan}). This eventually leads to a series expansion for spherical functions, known as the Harish-Chandra series.

Similar to how Dunkl theory generalizes the theory of spherical functions on Cartan motion groups, Heckman-Opdam theory generalizes the theory of spherical functions on Riemannian symmetric spaces of the non-compact type. Stripping $A$ in the Cartan decomposition $G=KAK$ of its ambient geometry (i.e. forgetting the $G$ and the $K$), we are left with a triple $(\Lie(A),\Sigma, m)$ of a Euclidean space endowed with a root system $\Sigma$ and multiplicities $m_{\alpha}, \alpha\in\Sigma$. Allowing the multiplicities $m_{\alpha}$ to be arbitrary complex parameters, one can still investigate the analogue of the algebra of radial parts. 
Studying this algebra, and in particular the analogue of the radial part of the Laplace-Beltrami operator, leads to a fruitful generalization of the theory of spherical functions. The joint eigenfunctions of this algebra are given by the hypergeometric function (\cite{heckmanOpdamI, heckmanII}). Later, this theory was further polished by the discovery of a curved analogue of the Dunkl operators, the Cherednik operators (\cite{cherednik}), which naturally give rise to this algebra of radial parts. Bounds for the hypergeometric function were studied in \cite{schapira}. Asymptotics of the hypergeometric function using its Harish-Chandra series expansion were obtained in \cite{narayananPasqualePusti}.

In the rank one case, the hypergeometric system of Heckman and Opdam reduces to the classical hypergeometric system and the Bessel system (\ref{eqn:jointEVProbBessel}) reduces to the classical Bessel system, see \cite{heckmanOpdamI} and \cite{dunkl1991, dunkl5}, respectively. It is already in this case that the Bessel system has an irregular singularity at $\infty$. Consequently, one cannot expect the usual machinery of Fuchsian differential equations to work and yield asymptotic expansions. Thus, already \cite{barletClerc, clercBessel} had to defer to different methods. 

\subsection{Absolute continuity in the geometric cases}

Absolute continuity, at least for the representing measure $\mu_{W, \lambda}^{k}=\frac{1}{|W|}\sum_{w\in W}\mu_{w\lambda}^{k}$ of the Bessel function, is immediate in the geometric cases due to the existence of an integral formula expressing $\psi_{\lambda}$ as an orbital integral, see \cite[p. 424, pp. 478f.]{helgasonGroups}. 
In the complex case $k=1$ (i.e. $G=K_{\C}$ is a complex Lie group), this measure is known as the Duistermaat-Heckman measure, and the density is known to be piecewise polynomial (\cite{heckmanProjectionsOfOrbits, duistermaatHeckman1982variation}).
An explicit formula for the density function can be obtained from \cite{graczykSawyerComplexProduct} and a classical relation between spherical functions of $G/K$ and $G_{0}/K$, see \cite{helgasonGroups}, Ch. IV, Prop. 4.10 and Ch. II, Thm. 3.15. Alternatively, one can employ the machinery of Hamiltonian actions, see \cite{guilleminLermanSternberg}.
The points of non-smoothness of the density lie on certain hyperplanes depending on $\lambda$.

\section{Appendix}\label{sec:appendix}
We now give the proofs of some technical lemmas we postponed earlier.

\subsection{Proof of \cref{lem:integrabilityB}}
      We first deal with the case $p=\xi$, so $\deg(p)=1$. 
      Recall $B=D_{1}+D_{2}+B_{2}-B_{3}-B_{4}$ from (\ref{eqn:systemPsi}) and the definition of $a=a(x_{0})\geq 1$ from \cref{eqn:ax0}.
      All terms in $B(s)$, except those of $D_{2}(s)$, are bounded by $C(a+s)^{-2}|\Omega_{I}^{-1/2}(x_{0})|$ and
      \begin{align*}
        \int_{0}^{\infty} (a+s)^{-2}\ ds \leq \int_{0}^{\infty} (1+s)^{-2}\ ds = \text{const.} < \infty. 
      \end{align*}
      It thus suffices to study the oscillating terms of $D_{2}(s)$. The entries of $D_{2}(s)$ are linear combinations of terms
      \begin{align*}
        f(s)=h(s)\Phi(s) 
      \end{align*}
      with 
      \begin{align*}
        h(s) = \frac{e^{-i\langle \alpha, x_{0}+sH_{i}\rangle\langle \alpha, g\lambda\rangle}}{\langle \alpha, x_{0}+sH_{i}\rangle}. 
      \end{align*}
      Since $\|h(s)\|\leq \text{const.}\cdot (a+s)^{-1}$ and $T$ is treated as a constant, we immediately obtain for $t\leq T$ that
      \begin{align*}
        \left\|\int_{0}^{t} D_{2}(s)\ ds\right\| \leq \text{const.}\cdot |\Omega_{I}^{-1/2}(x_{0})|
      \end{align*}
      by the estimate on $\Phi=\omega_{I}^{1/2}\wt\Phi$ of our induction hypothesis (\cref{thm:mainthm}). This finishes the proof of (1).
      With the same reasoning as in \cref{prop:AQentries}, we see that
      \begin{align*}
        H(t) \coloneqq -\int_{t}^{\infty} h(s)\ ds
      \end{align*}
      converges and that $|H(t)|\leq \text{const.}\cdot(a+t)^{-1}$. We now denote
      \begin{align*}
        M(t) \coloneqq (\mathbbm{1}+Q(t))^{-1} 
      \end{align*}
      and see $M'(t) = -M(t)A(t)M(t)$, as $Q'=A$. Thus integration by parts gives
      \begin{align}\label{eqn:integrabilityBintByParts}
        &\int_{T}^{t} M(s)h(s)\Phi(s)\ ds\nonumber \\
        =& \big[M(s)\Phi(s)H(s)\big]_{s=T}^{t} - \int_{T}^{t} (M(s)\Phi(s))' H(s)\ ds\nonumber\\
        =& \big[M(s)\Phi(s)H(s)\big]_{s=T}^{t} - \int_{T}^{t} \Big(-M(s)A(s)M(s)\Phi(s)+ M(s)\Phi'(s)\Big) H(s)\ ds\\
        =& \big[M(s)\Phi(s)H(s)\big]_{s=T}^{t} - \int_{T}^{t} \Big(-M(s)A(s)M(s)\Phi(s)+ M(s)A(s)\Phi(s)\Big) H(s)\ ds.\nonumber
      \end{align}
      But $\|A(s)\|\leq \text{const.}\cdot (a+s)^{-1}$ and $|H(s)|\leq \text{const.}\cdot (a+s)^{-1}$ and the $M(s)$ are uniformly bounded by \cref{lem:QBound}. Thus, the integrand is bounded by $\text{const.}\cdot (a+s)^{-2}|\Omega_{I}^{-1/2}(x_{0})|$.
      We conclude
      \begin{align*}
        \left\|\int_{T}^{t} M(s)h(s)\Phi(s)\ ds\right\| \leq \text{const.}\cdot |\Omega_{I}^{-1/2}(x_{0})|,
      \end{align*}
      which finishes the proof of (2).

      This proves the lemma in the case $\deg(p)=1$. 
      We now turn to the general case of $\deg(p)\geq 1$.
      For $p=\xi_{1}\ldots\xi_{\ell}$, using the Leibniz rule and $\partial_{H}(\omega_{I}^{1/2}\wt\Phi)=A\omega_{I}^{1/2}\wt\Phi$, we see
      \begin{align*}
        \partial_{H}\Psi = A\Psi + \sum_{q_{1}, q_{2}, q_{3}} \big(q_{1}(\partial)A\big) \big(q_{2}(\partial)\omega_{I}^{1/2}\big) \big(q_{3}(\partial)\wt\Phi\big) = A\Psi + B
      \end{align*}
      with certain $q_{1},q_{2},q_{3}$ with $\deg(q_{1})+\deg(q_{2})+\deg(q_{3})=\deg(p)$ and $\Psi = \omega_{I}^{1/2} p(\partial)\wt\Phi$. We can study the different terms of $B$ in a similar fashion as we did after (\ref{eqn:systemPsi}). 
      Note that whenever $\deg(q_{3})=\deg(p)$, there is a $\partial_{H}$ appearing in $q_{3}(\partial)$, which enables one to fall back to the Dunkl eigenvalue property, as was done for the term $B_{4}(x)$ in (\ref{eqn:systemPsi}). Thus for expressions $\omega_{I}^{1/2}q_{3}(\partial)\wt\Phi$ one can invoke estimates from the induction hypothesis.

Indeed, most of the terms in $B$ will be absolutely integrable, like we already saw in the discussion after (\ref{eqn:systemPsi}). The discussion of those terms is easy. Interesting are the oscillating terms, similar to $D_{2}$ above. Given $q_{1}, q_{3}\in\C[\a]$ with $\deg(q_{i})<\deg(p)$, the expression $\big(q_{1}(\partial)A\big)\big(\omega_{I}^{1/2}q_{3}(\partial) \wt\Phi\big)$ contains linear combinations of terms of the form
  \begin{align*}
    k_{\alpha} \frac{e^{-i\langle \alpha, x\rangle \langle \alpha, g\lambda\rangle}}{\langle \alpha, x\rangle} \big(\omega_{I}^{1/2}q_{3}(\partial)\wt\Phi\big)
  \end{align*}
  in its entries. While we used $\Phi'=A\Phi$ for $\Phi=\omega_{I}^{1/2}\wt\Phi$ in the proof of (2) above, more precisely in \cref{eqn:integrabilityBintByParts}, the proof still works in this situation: By the induction hypothesis for $\deg(q_{3})<\deg(p)$, $\omega_{I}^{1/2}q_{3}(\partial)\wt\Phi$ also satisfies a system of the form (\ref{eqn:systemDeriv}). As $|A(t)H(t)|\leq \text{const.}\cdot (a+t)^{-2}$ and $|B(t)H(t)|\leq \text{const.}\cdot(a+t)^{-2}|\Omega_{I}^{-1/2}(x_{0})|$, the proof goes through.
  In the case $\deg(q_{3}) = 0$ discussed previously, the additive perturbation in (\ref{eqn:systemDeriv}) was just $0$.

  \qed

  \subsection{Proof of \cref{lem:constructionQcomplex}}
    To construct $Q$, we argue as in the proof of \cite[Thm. 4.26]{bodineLutz}. We reiterate certain parts of the proof for the reader's convenience.
    Note that the matrix $\Gamma$ consists only of real entries, which we denote $\Gamma_{w}=\langle H, w\lambda_{1}\rangle$. Thus $\Gamma=\diag(\Gamma_{w})_{w\in W}$.
    Define $\epsilon\geq 0$ as
    \begin{align*}
      \epsilon \coloneqq \min\{|\Gamma_{g}-\Gamma_{h}| : g,h\in W \text{ s.t. } \Gamma_{g}\neq \Gamma_{h}\},
    \end{align*}
    with the convention $\epsilon=0$ if $\Gamma_{g}=\Gamma_{h}$ for all $g, h\in W$. 
    \begin{sublemma}\label{sublemma:rho}
      If $\epsilon>0$, then there exist constants $0\leq \rho< \epsilon$ and $t_{0}\geq 0$ such that $e^{\rho t}(a+t)^{-1}$ is non-decreasing for all $t\geq t_{0}$. The constant $t_{0}$ does not depend on $a\geq 1$.
    \end{sublemma}
    \begin{proof}[Proof of Sublemma]
       We have 
       \begin{align*}
         \frac{d}{dt} e^{\rho t}(a+t)^{-1} = e^{\rho t}(a+t)^{-2} (a\rho+\rho t -1)\geq 0\quad \Longleftrightarrow\quad a\rho+\rho t \geq 1. 
       \end{align*}
       Since $a\geq 1$, this is certainly true for $\rho > (1+t)^{-1}$. In particular, there exists a $t_{0}\geq 0$ such that $(1+t)^{-1}\leq (1+t_{0})^{-1}<\epsilon/2$ for all $t\geq t_{0}$.
    \end{proof}

    We fix $t_{0}$ and $\rho$ as above if $\epsilon>0$ and continue with the proof of \cref{lem:constructionQcomplex}. 
    We define 
    \begin{align*}
      P(t) \coloneqq -\int_{t}^{\infty} A(s)\ ds, 
    \end{align*}
    which converges by \cref{prop:AQentries} and satisfies $P'=A$ (this was our previous $Q$). In fact, $P\in\mathcal O((a+t)^{-1})$ by \cref{lem:estNonParabRootsBetter}.
  The condition $Q'=[\Gamma, Q]+A$ translates to
  \begin{align}\label{eqn:odeQ}
    Q_{g, h}' = (\Gamma_{g}-\Gamma_{h})Q_{g, h} + A_{g, h}.
  \end{align}
  We consider several cases:
  \begin{enumerate}
    \item If $\Gamma_{g}-\Gamma_{h}>0$, we define
      \begin{align}\label{eqn:defQgeq0}
        Q_{g, h}(t) \coloneqq -\int_{t}^{\infty} e^{(\Gamma_{g}-\Gamma_{h})(t-s)}A_{g, h}(s)\ ds. 
      \end{align}
      Then $Q_{g, h}$ is a solution to (\ref{eqn:odeQ}) as one immediately verifies. 
      Integration by parts shows that
      \begin{align*}
        Q_{g, h}(t)=P_{g, h}(t)-(\Gamma_{g}-\Gamma_{h})\int_{t}^{\infty} e^{(\Gamma_{g}-\Gamma_{h})(t-s)}P_{g, h}(s)\ ds. 
      \end{align*}
      Now $(\Gamma_{g}-\Gamma_{h})(t-s)<\epsilon(t-s)$ for all $s\geq t$ and we already saw $P_{g, h}(s)\in \mathcal O((a+s)^{-1})$. Since furthermore $(a+s)^{-1}$ is monotonically decreasing in $s$, we conclude that
      \begin{align*}
        |Q_{g, h}(t)| &\leq |P_{g, h}(t)| + \text{const.}\cdot \int_{t}^{\infty} e^{\epsilon (t-s)}(a+s)^{-1}\ ds\\
                      &\leq \text{const.}\cdot (a+t)^{-1}\left(1+\int_{t}^{\infty}e^{\epsilon (t-s)}\ ds\right).
      \end{align*}
      Thus $Q_{g, h}\in \mathcal O((a+t)^{-1})$.

    \item  If $\Gamma_{g}-\Gamma_{h}=0$, we define $Q_{g, h}$ as in \cref{eqn:defQgeq0}. Now $Q_{g, h}(t)\in\mathcal O((a+t)^{-1})$ is immediate from $P_{g, h}(t)\in\mathcal O((a+t)^{-1})$.

    \item If $\Gamma_{g}-\Gamma_{h}<0$, we instead define
      \begin{align}\label{eqn:defQless0}
        Q_{g, h}(t)\coloneqq \int_{t_{0}}^{t} e^{(\Gamma_{g}-\Gamma_{h})(t-s)}A_{g,h}(s)\ ds.
      \end{align}
      Again, $Q_{g, h}$ satisfies \cref{eqn:odeQ}. In order to establish $Q_{g, h}(t)\in\mathcal O((a+t)^{-1})$, \cref{sublemma:rho} will be important. Indeed, integration by parts shows that
      \begin{align}\label{eqn:Qless0IntByParts}
        Q_{g, h}(t) = P_{g, h}(t)-e^{(\Gamma_{g}-\Gamma_{h})(t-t_{0})}P_{g, h}(t_{0}) +\int_{t_{0}}^{t} (\Gamma_{g}-\Gamma_{h})e^{(\Gamma_{g}-\Gamma_{h})(t-s)}P_{g, h}(s)\ ds. 
      \end{align}
      Now $(\Gamma_{g}-\Gamma_{h})(t-t_{0})<-\epsilon(t-t_{0})$ for $t\geq t_{0}$ and by \cref{sublemma:rho} 
      \begin{align*}
        (a+s)^{-1} < e^{\rho (t-s)}(a+t)^{-1}\qquad (\forall t\geq s\geq t_{0}).
      \end{align*}
      Thus, we obtain from \cref{eqn:Qless0IntByParts} that
      \begin{align*}
        |Q_{g, h}(t)| &\leq \text{const.}\cdot \left((a+t)^{-1} + e^{(\rho-\epsilon)(t-t_{0})}(a+t)^{-1} + (a+t)^{-1}\int_{t_{0}}^{t} e^{(\rho-\epsilon)(t-s)}\ ds\right)\\
                      &\leq \text{const.}\cdot (a+t)^{-1}.
      \end{align*}
  \end{enumerate}
  Thus, we constructed a function $Q$ with the desired properties.

  \qed

\section*{Acknowledgements}
We would like to thank Margit Rösler for suggesting this topic and for many helpful discussions and comments on the manuscript. We would further like to thank Dominik Brennecken for many helpful discussions during the early stages of this project.
\fundingThanks


\begin{thebibliography}{NPP14}

\bibitem[AG21]{amriGasmi}
B.~Amri and A.~Gasmi.
\newblock On the estimates of the {D}unkl {K}ernel.
\newblock {\em Anal. Math.}, 47:1--12, 2021.

\bibitem[AT25]{ankerTrojanDihedral}
J.-P. Anker and B.~Trojan.
\newblock Optimal bounds for the {D}unkl kernel in the dihedral case.
\newblock {\em J. Funct. Anal.}, 288(3):110743, 2025.

\bibitem[BC86]{barletClerc}
D.~Barlet and J.-L. Clerc.
\newblock Le comportement à l'infini des fonctions de {B}essel généralisées, {I}.
\newblock {\em Adv. Math.}, 61(2):165--183, 1986.

\bibitem[BL15]{bodineLutz}
S.~Bodine and D.~A. Lutz.
\newblock {\em Asymptotic {I}ntegration of {D}ifferential and {D}ifference {E}quations}.
\newblock Lecture Notes in Mathematics : 2129. Cham : Springer International Publishing, 2015.

\bibitem[Che55]{chevalley}
C.~Chevalley.
\newblock Invariants of {F}inite {G}roups {G}enerated by {R}eflections.
\newblock {\em Amer. J. Math.}, 77(4):778--782, 1955.

\bibitem[Che91]{cherednik}
I.~Cherednik.
\newblock A unification of {K}nizhnik-{Z}amolodchikov and {D}unkl operators via affine {H}ecke algebras.
\newblock {\em Invent. Math.}, 106(2):411--432, 1991.

\bibitem[Cle87]{clercBessel}
J.-L. Clerc.
\newblock Le comportement à l'infini des fonctions de {B}essel généralisées, {II}.
\newblock {\em Adv. Math.}, 66(1):31--61, 1987.

\bibitem[CM82]{casselmanMilicic}
W.~Casselman and D.~Miličić.
\newblock {Asymptotic behavior of matrix coefficients of admissible representations}.
\newblock {\em Duke Math. J.}, 49(4):869 -- 930, 1982.

\bibitem[DH82]{duistermaatHeckman1982variation}
J.~J. Duistermaat and G.~J. Heckman.
\newblock On the variation in the cohomology of the symplectic form of the reduced phase space.
\newblock {\em Invent. Math.}, 69(2):259--268, 1982.

\bibitem[DH23]{dziubanskiHejna}
J.~Dziuba{\'n}ski and A.~Hejna.
\newblock Upper and lower bounds for the {D}unkl heat kernel.
\newblock {\em Calc. Var. Partial Differ. Equ.}, 62(1):25, 2023.

\bibitem[dJ93]{dejeu}
M.~F.~E. de~Jeu.
\newblock The {D}unkl transform.
\newblock {\em Invent. Math.}, 113(1):147--162, 1993.

\bibitem[dJ06]{deJeuPaleyWiener}
M.~F.~E. de~Jeu.
\newblock Paley--{W}iener theorems for the {D}unkl transform.
\newblock {\em Trans. Amer. Math. Soc.}, 358(10):4225--4250, 2006.

\bibitem[Dun89]{dunkl2}
C.~F. Dunkl.
\newblock Differential-{D}ifference {O}perators {A}ssociated to {R}eflection {G}roups.
\newblock {\em Trans. Amer. Math. Soc.}, 311:167--183, 1989.

\bibitem[Dun90]{dunkl3}
C.~F. Dunkl.
\newblock Operators commuting with {C}oxeter group actions on polynomials.
\newblock In D.~Stanton, editor, {\em Invariant {T}heory and {T}ableaux}, pages 107--117. Springer, 1990.

\bibitem[Dun91]{dunkl1991}
C.~F. Dunkl.
\newblock Integral {K}ernels with {R}eflection {G}roup {I}nvariance.
\newblock {\em Canad. J. Math.}, 43(6):1213--1227, 1991.

\bibitem[Dun92]{dunkl5}
C.~F. Dunkl.
\newblock Hankel transforms associated to finite reflection groups.
\newblock {\em Contemp. Math.}, 138:123--138, 1992.

\bibitem[Eas89]{eastham}
M.~S.~P. Eastham.
\newblock {\em The {A}symptotic {S}olution of {L}inear {D}ifferential {S}ystems: {A}pplications of the {L}evinson theorem.}
\newblock London Mathematical Society monographs : NS,4, Oxford science publications. Clarendon Press, 1989.

\bibitem[ER25]{etingofRains}
P.~Etingof and E.~Rains.
\newblock Bounds for asymptotic characters of simple {L}ie groups.
\newblock {\em Indag. Math.}, 36(6):1572--1591, 2025.
\newblock Special Issue: Dedicated to Tom Koornwinder on the occasion of his 80th birthday.

\bibitem[GLS96]{guilleminLermanSternberg}
V.~Guillemin, E.~Lerman, and S.~Sternberg.
\newblock {\em Symplectic fibrations and multiplicity diagrams}.
\newblock Cambridge University Press, 1996.

\bibitem[GS02]{graczykSawyerComplexProduct}
P.~Graczyk and P.~Sawyer.
\newblock The product formula for the spherical functions on symmetric spaces in the complex case.
\newblock {\em Pacific J. Math.}, 204(2):377--393, 2002.

\bibitem[GS23]{graczykSawyerAn}
P.~Graczyk and P.~Sawyer.
\newblock Sharp estimates for {$\uppercase{W}$}-invariant {D}unkl and heat kernels in the $\uppercase{A}_{n}$ case.
\newblock {\em Bull. Sci. Math.}, 186:103271, 2023.

\bibitem[GS24]{graczykSawyerA2}
P.~Graczyk and P.~Sawyer.
\newblock A {F}ormula and {S}harp {E}stimates for the {D}unkl {K}ernel for the {R}oot {S}ystem $\uppercase{A}_{2}$.
\newblock {\em J. Lie Theory}, 34(3):577--594, 2024.

\bibitem[GV88]{gangolliVaradarajan}
R.~Gangolli and V.~S. Varadarajan.
\newblock {\em Harmonic {A}nalysis of {S}pherical {F}unctions on {R}eal {R}eductive {G}roups}.
\newblock Ergebnisse der Mathematik und ihrer Grenzgebiete : B,101. Berlin [u.a.] : Springer, 1988.

\bibitem[HC57]{hcDiffOps}
Harish-Chandra.
\newblock {D}ifferential {O}perators on a {S}emisimple {L}ie {A}lgebra.
\newblock {\em Amer. J. Math.}, 79(1):87--120, 1957.

\bibitem[HC58a]{hcSphericalI}
Harish-Chandra.
\newblock Spherical {F}unctions on a {S}emisimple {L}ie {G}roup {I}.
\newblock {\em Amer. J. Math.}, 80(3):241--310, 1958.

\bibitem[HC58b]{hcSpherical}
Harish-Chandra.
\newblock Spherical {F}unctions on a {S}emisimple {L}ie {G}roup {II}.
\newblock {\em Amer. J. Math.}, 80(3):553--613, 1958.

\bibitem[Hec82]{heckmanProjectionsOfOrbits}
G.~J. Heckman.
\newblock Projections of orbits and asymptotic behavior of multiplicities for compact connected {L}ie groups.
\newblock {\em Invent. Math.}, 67:333--356, 1982.

\bibitem[Hec87]{heckmanII}
G.~J. Heckman.
\newblock Root {S}ystems and {H}ypergeometric {F}unctions. {II}.
\newblock {\em Compos. Math.}, 64(3):353--373, 1987.

\bibitem[Hec89]{heckman}
G.~J. Heckman.
\newblock A remark on the {D}unkl differential-difference operators.
\newblock In {\em Harmonic analysis on reductive groups}, volume 101 of {\em Progr. Math.}, pages 181--191. Birkhäuser, 1989.

\bibitem[Hel00]{helgasonGroups}
S.~Helgason.
\newblock {\em Groups and Geometric Analysis: Integral Geometry, Invariant Differential Operators, and Spherical Functions}.
\newblock Mathematical surveys and monographs. American Mathematical Society, 2000.

\bibitem[HO87]{heckmanOpdamI}
G.~J. Heckman and E.~M. Opdam.
\newblock Root {S}ystems and {H}ypergeometric {F}unctions. {I}.
\newblock {\em Compos. Math.}, 64(3):329--352, 1987.

\bibitem[NPP14]{narayananPasqualePusti}
E.~K. Narayanan, A.~Pasquale, and S.~Pusti.
\newblock Asymptotics of {H}arish-{C}handra expansions, bounded hypergeometric functions associated with root systems, and applications.
\newblock {\em Adv. Math.}, 252:227--259, 2014.

\bibitem[Opd93]{opdamBessel}
E.~M. Opdam.
\newblock Dunkl operators, {Bessel} functions and the discriminant of a finite {Coxeter} group.
\newblock {\em Compos. Math.}, 85(3):333--373, 1993.

\bibitem[RdJ02]{roslerDeJeu}
M.~Rösler and M.~F.~E. de~Jeu.
\newblock Asymptotic analysis for the {D}unkl kernel.
\newblock {\em J. Approx. Theory}, 119(1):110--126, 2002.

\bibitem[Rö99]{roslerPositivity}
M.~Rösler.
\newblock Positivity of {D}unkl's {I}ntertwining {O}perator.
\newblock {\em Duke Math. J.}, 98(3):445--463, 1999.

\bibitem[Saw17]{sawyer2017laplace}
P.~Sawyer.
\newblock A {L}aplace-type representation of the generalized spherical functions associated with the root systems of type {$A$}.
\newblock {\em Mediterr. J. Math.}, 14(4):147, 2017.

\bibitem[Sch08]{schapira}
B.~Schapira.
\newblock Contributions to the hypergeometric function theory of {H}eckman and {O}pdam: sharp estimates, {S}chwartz space, heat kernel.
\newblock {\em Geom. Funct. Anal.}, 18(1):222--250, 2008.

\bibitem[Tit39]{titchmarsh}
E.~C. Titchmarsh.
\newblock {\em The theory of functions}.
\newblock Oxford Univ. Press, London, 2nd edition, 1939.

\bibitem[Tre67]{treves}
F.~Treves.
\newblock {\em Topological {V}ector {S}paces, {D}istributions and {K}ernels}.
\newblock Pure and Applied Mathematics. Academic Press, 1967.

\bibitem[TV71]{trombiVaradarajan}
P.~C. Trombi and V.~S. Varadarajan.
\newblock Spherical {T}ransforms on {S}emisimple {L}ie {G}roups.
\newblock {\em Ann. of Math.}, 94(2):246--303, 1971.

\end{thebibliography}
\end{document}